\documentclass[11pt,letterpaper]{amsart}
\usepackage{amsfonts}
\usepackage{amsmath}
\usepackage{amsthm}
\usepackage{amssymb}
\usepackage{amsmath,amscd}
\usepackage{graphicx}
\usepackage{enumerate}
\usepackage{subcaption}
\usepackage[pdffitwindow=false,
colorlinks=true,linkcolor=black,urlcolor=black,citecolor=black]{hyperref}
\usepackage{tabularx}
\usepackage{float}
\usepackage{float}
\usepackage{multirow}

\newtheorem{theorem}{Theorem}[section]
\newtheorem{lemma}[theorem]{Lemma}
\theoremstyle{definition}

\newtheorem{proposition}[theorem]{Proposition}

\newtheorem{remark}[theorem]{Remark}
\numberwithin{equation}{section}
\newcommand{\spin}{\mathrm{Spin}}
\newcommand{\trace}{\mathrm{tr}}
\newcommand{\diag}{\mathrm{diag}}

\title{Cohomogeneity one $\spin(7)$ metrics with generic Aloff--Wallach spaces as principal orbits}
\author{Hanci Chi}
\address{Department of Pure Mathematics\\ Xi'an Jiaotong-Liverpool University\\ Suzhou 215123\\ China}
\email{hanci.chi@xjtlu.edu.cn}
\begin{document}
\maketitle
\begin{abstract}
This paper establishes the existence of forward complete cohomogeneity one $\mathrm{Spin}(7)$ metrics with generic Aloff--Wallach spaces $N_{k,l}$ as principal orbits and $\mathbb{CP}^2$ as the singular orbit, building on Reidegeld’s analysis of the initial value problem. We construct three continuous one-parameter families of non-compact $\mathrm{Spin}(7)$ metrics. Each family contains a limiting asymptotically conical (AC) metric, while the other metrics in the families are asymptotically locally conical (ALC). Moreover, two of the AC metrics share the same asymptotic cone, exhibiting a geometric transition phenomenon analogous to that found by Lehmann in the exceptional case.
\end{abstract}
\let\thefootnote\relax\footnote{2020 Mathematics Subject Classification: 53C25 (primary).

Keywords: $\spin(7)$ metric, cohomogeneity one metric.

The author is supported by the NSFC (No. 12301078), the Natural Science Foundation of Jiangsu Province (BK-20220282), and the XJTLU Research Development Funding (RDF-21-02-083).}

\section{Introduction}
Metrics with $\spin(7)$ holonomy are of great interest in both differential geometry and theoretical physics. The first example was constructed in \cite{bryant_metrics_1987} on the cone over the Berger space $SO(5)/SO(3)$. The first complete example was constructed in \cite{bryant_construction_1989}, defined on the spinor bundle over $\mathbb{S}^4$. The first compact example was given in \cite{joyce_new_1999}. The compact examples rely on resolving orbifold singularities followed by delicate analytic perturbation, whereas most non-compact examples are obtained by exploiting symmetry, which reduces the $\spin(7)$ equations to a system of ODEs. These examples are of cohomogeneity one.

In this paper, we follow the latter approach to seek further examples of $\spin(7)$ metrics, where the geometry is determined by the choice of principal orbit. Among the possible homogeneous 7-manifolds, Aloff--Wallach spaces stand out as ideal candidates for constructing new $\spin(7)$ metrics. An Aloff--Wallach space is a homogeneous space $N_{k,l}:=SU(3)/U(1)_{k,l}$, where $U(1)_{k,l}$ is embedded in $SU(3)$ as $$\diag\left(e^{k\sqrt{-1}t},e^{l\sqrt{-1}t},e^{-(k+l)\sqrt{-1}t}\right).$$ Without loss of generality, we set $(k,l)$ to be non-negative and coprime. Due to the geometric differences, an Aloff--Wallach space is \emph{exceptional} if  $kl(k-l)= 0$ and \emph{generic} if otherwise. Aloff--Wallach spaces possess remarkably rich topological and geometric structures. Setting $\Delta=k^2+kl+l^2$, we have 
$$
H^4(N_{k,l},\mathbb{Z})=\mathbb{Z}/\Delta\mathbb{Z},
$$
which shows that there exist infinitely many homotopy types among them. Moreover, there are known examples of homeomorphic but non-diffeomorphic Aloff--Wallach spaces \cite{kreck1991some}. An Aloff--Wallach space can be regarded as a principal $\mathbb{S}^1$-bundle over $SU(3)/T^2$, whose bundle structure depends on the chosen embedding of $U(1)_{k,l}\subset SU(3)$. Alternatively, each $N_{k,l}$ can also be viewed as lens space bundles
$$\mathbb{S}^3/\mathbb{Z}_i\hookrightarrow N_{k,l}\rightarrow \mathbb{CP}^2,\quad i\in\{k + l,l,k\},$$
giving rise to three topologically distinct cohomogeneity one orbifold $M_{k,l}^{i}$, each as an $\mathbb{R}^4/\mathbb{Z}_i$-bundle over $\mathbb{CP}^2$. By \cite{wang_some_1982, castellani1984n, page1984new, kowalski_homogeneous_1993}, each generic Aloff--Wallach space admits two non-isometric $SU(3)$-invariant Einstein metrics. The Euclidean metric cones over these homogeneous Einstein metrics are $\spin(7)$ \cite{bar1993real}. This naturally motivates the study of cohomogeneity one $\spin(7)$ metrics on $M_{k,l}^i$.

For the exceptional cases, the picture is largely understood. For $N_{1,1}$, the orbifold $M_{1,1}^2$ admits the Bryant--Salamon metric in \cite{bryant_construction_1989}, which belongs to the $\mathbb{B}_8$ family introduced numerically in \cite{cvetic_cohomogeneity_2002} and later proved to exist in \cite{bazaikin_new_2007,bazaikin_noncompact_2008}. The space $M_{1,1}^1$ is the manifold $T^*\mathbb{CP}^2$, and it carries the Calabi metric \cite{calabi1979metriques}, which is HyperK\"ahler. The metric appears as the limiting AC metric of the one-parameter family of ALC $\spin(7)$ metrics constructed in \cite{chi2022spin}. For $N_{1,0}$, it was conjectured in \cite{cvetic_cohomogeneity_2002} and \cite{gukov2003conifold} that there exist two topologically distinct resolutions of the same $\spin(7)$ cone. The conjecture was later confirmed in \cite{lehmann2022geometric}, where two topologically different continuous families of ALC $\spin(7)$ metrics were constructed: one has singular orbit $\mathbb{S}^5$, and the other has $\mathbb{CP}^2$. Each family contains an AC limiting metric, and the two AC metrics share the same asymptotic cone based on the unique $SU(3)\times U(1)$-invariant nearly parallel $G_2$ structure on $N_{1,0}$.

In contrast, the situation for generic Aloff--Wallach spaces remains less understood. By Reidegeld’s local analysis in \cite{reidegeld2011exceptional}, a singular orbit as $\mathbb{S}^5$ can only occur in the $N_{1,0}$ case. The main interest lies in forward complete examples where $N_{k,l}$ collapses to $\mathbb{CP}^2$. Explicit isolated examples were obtained in \cite{cvetic_cohomogeneity_2002, kanno_spin7_2002-1}. The generic case was further investigated in \cite{chi2022spin}, where we proved the existence of a continuous one-parameter family of ALC $\spin(7)$ metrics on $ M_{k,l}^{k+l}$ and $ M_{k,l}^{k}$ with an artificial assumption $k> l$.

The present paper serves as a sequel to \cite{chi2022spin} in two aspects. Firstly, we reconstruct an invariant set that applies to all coprime pairs $(k,l)$. Being derived from comparing the metric components of the $\mathbb{S}^1$-fiber and the $SU(3)/T^2$ base, the new invariant set is more geometrically motivated. Secondly, we construct another invariant set, where the $\mathbb{S}^1$-fiber blows up. The limiting integral curves that stay between these two invariant sets represent AC $\spin(7)$ metrics, which are analogous to the Lehmann's AC metrics in the $N_{1,0}$ case. 

\begin{theorem}
\label{thm: main 1}
For any coprime pair $(k,l)$, on each $M_{k,l}^{i}$ with $i\in \{k+l,l,k\}$, there exists a continuous one-parameter family of forward complete 
$\spin(7)$ metrics $$\left\{\gamma^i_{\theta}\mid \theta\in (0,\theta_i]\right\}.$$
For $\theta\in(0,\theta_i)$, the metric $\gamma^i_{\theta}$ is asymptotically locally conical (ALC), with its ALC limit modeled on an $ \mathbb{S}^1 $-bundle over the $G_2$ cone on the nearly K\"ahler $ SU(3)/T^2$. For the endpoint $\theta_i$, the limiting metric $\gamma^i_{\theta_i}$ is asymptotically conical (AC). The AC limits of $\gamma^k_{\theta_k}$ and  $\gamma^l_{\theta_l}$ are modeled on a homogeneous Einstein metric on $N_{k,l}$, whereas the AC limit of $\gamma^{k+l}_{\theta_{k+l}}$ is modeled on another homogeneous Einstein metric on $N_{k,l}$.
\end{theorem}

The spirit of Lehmann’s metrics is reflected in the main theorem above. For each coprime pair $(k,l)$, the orbifolds $M_{k,l}^k$ and $M_{k,l}^l$	are topologically distinct, both carrying AC $\spin(7)$ metrics that share the same asymptotic cone. In particular, if $1\in\{k,l\}$, one of the resolutions of the $\spin(7)$ cone on $N_{k,l}$ yields a smooth manifold $M_{k,l}^1$, whereas the other still has an orbifold singularity.

This paper is organized as follows. In Section \ref{sec: Ricci-flat Equation}, we briefly recall the setup and main equations from our previous work \cite{chi2022spin}. Section \ref{sec: local analysis} reviews the local analysis developed therein. These two sections are included mainly for completeness, and readers familiar with the subject may safely skip them. In Section \ref{sec: invariant set 1}, we construct invariant sets that lead to the existence of one-parameter families of ALC $\spin(7)$ metrics on each $M_{k,l}^i$. In Section \ref{sec: invariant set 2}, we construct another invariant set, from which the existence of AC $\spin(7)$ metrics on $M_{k,l}^i$ follows.

%

\section{The $\spin(7)$ holonomy cohomogeneity one system}
\label{sec: Ricci-flat Equation}
The cohomogeneity one Ricci-flat equations derived in \cite{reidegeld2011exceptional} were reformulated in \cite{chi2022spin} as a dynamical system in the variables
$$(X_1,X_2,X_3,X_4,Z_1,Z_2,Z_3,Z_4)$$ constrained to an algebraic surface. Here the variables $X_j$ represent the normalized principal curvatures, while the variables $Z_j$ correspond normalized metric components. In the following, we give a brief presentation of the dynamical system and the definitions of $(X_j,Z_j)$. For full details and derivations, the reader is referred to \cite{reidegeld2011exceptional, chi2022spin}.

For an Aloff–Wallach space $N_{k,l}$, we fix a basis for $\mathfrak{su}(3)$ as in \cite[(4.3)]{reidegeld2011exceptional}. The isotropy representation $\mathfrak{su}(3)/\mathfrak{u}(1)_{k,l}$ is decomposed as 
\begin{equation}
\label{eqn_isotropy representation}
\mathfrak{su}(3)/\mathfrak{u}(1)_{k,l}=\mathbf{2}_{k-l}\oplus  \mathbf{2}_{2k+l}\oplus  \mathbf{2}_{k+2l}\oplus \mathbf{1},
\end{equation}
where each subscript denotes the corresponding $\mathfrak{u}(1)_{k,l}$-weight. With our convention that $k$ and $l$ are non-negative coprime integers, the above four irreducible summands are pairwise inequivalent if and only if 
$$(k,l)\notin\{(1,0),(0,1),(1,1)\},$$ 
which is the generic case for Aloff–Wallach spaces. By the $SU(3)$-action, an invariant metric $g_{N_{k,l}}$ on the principal orbit is determined by a positive definite symmetric bilinear form on $\mathfrak{su}(3)/\mathfrak{u}(1)_{k,l}$. Hence, the matrix representation of $g_{N_{k,l}}$ takes the form
\begin{equation}
\label{eqn: invariant metric}
\begin{bmatrix}
a^2&0&&&&&\\
0&a^2&&&&&&\\
&&b^2&0&&&\\
&&0&b^2&&&\\
&&&&c^2&0&\\
&&&&0&c^2&\\
&&&&&&f^2
\end{bmatrix},
\end{equation}
where $a^2$, $b^2$, $c^2$, and $f^2$ respectively correspond to irreducible summands in \eqref{eqn_isotropy representation}.
Consider a cohomogeneity one metric of the form 
$$g=dt^2+g_{N_{k,l}}(t),$$ where components in \eqref{eqn: invariant metric} are now functions of $t$. We follow \cite{eschenburg_initial_2000, reidegeld2011exceptional} to obtain the cohomogeneity one Ricci-flat system:
\begin{equation}
\label{eqn: original Einstein equation}
\begin{split}
\frac{\ddot{a}}{a}-\left(\frac{\dot{a}}{a}\right)^2&=-\left(2\frac{\dot{a}}{a}+2\frac{\dot{b}}{b}+2\frac{\dot{c}}{c}+\frac{\dot{f}}{f}\right)\frac{\dot{a}}{a}+\frac{6}{a^2}+\frac{a^2}{b^2c^2}-\frac{b^2}{a^2c^2}-\frac{c^2}{a^2b^2}-\frac{1}{2}\frac{(k+l)^2}{\Delta^2}\frac{f^2}{a^4},\\
\frac{\ddot{b}}{b}-\left(\frac{\dot{b}}{b}\right)^2&=-\left(2\frac{\dot{a}}{a}+2\frac{\dot{b}}{b}+2\frac{\dot{c}}{c}+\frac{\dot{f}}{f}\right)\frac{\dot{b}}{b}+\frac{6}{b^2}+\frac{b^2}{a^2c^2}-\frac{c^2}{a^2b^2}-\frac{a^2}{b^2c^2}-\frac{1}{2}\frac{l^2}{\Delta^2}\frac{f^2}{b^4},\\
\frac{\ddot{c}}{c}-\left(\frac{\dot{c}}{c}\right)^2&=-\left(2\frac{\dot{a}}{a}+2\frac{\dot{b}}{b}+2\frac{\dot{c}}{c}+\frac{\dot{f}}{f}\right)\frac{\dot{c}}{c}+\frac{6}{c^2}+\frac{c^2}{a^2b^2}-\frac{a^2}{b^2c^2}-\frac{b^2}{a^2c^2}-\frac{1}{2}\frac{k^2}{\Delta^2}\frac{f^2}{c^4},\\
\frac{\ddot{f}}{f}-\left(\frac{\dot{f}}{f}\right)^2&=-\left(2\frac{\dot{a}}{a}+2\frac{\dot{b}}{b}+2\frac{\dot{c}}{c}+\frac{\dot{f}}{f}\right)\frac{\dot{f}}{f}+\frac{1}{2}\frac{(k+l)^2}{\Delta^2}\frac{f^2}{a^4}+\frac{1}{2}\frac{l^2}{\Delta^2}\frac{f^2}{b^4}+\frac{1}{2}\frac{k^2}{\Delta^2}\frac{f^2}{c^4},
\end{split}
\end{equation}
with a conservation law
\begin{equation}
\label{eqn: original conservation}
\begin{split}
&\left(2\frac{\dot{a}}{a}+2\frac{\dot{b}}{b}+2\frac{\dot{c}}{c}+\frac{\dot{f}}{f}\right)^2-2\left(\frac{\dot{a}}{a}\right)^2-2\left(\frac{\dot{b}}{b}\right)^2-2\left(\frac{\dot{c}}{c}\right)^2-\left(\frac{\dot{f}}{f}\right)^2\\
&=12\left(\frac{1}{a^2}+\frac{1}{b^2}+\frac{1}{c^2}\right)-2\left(\frac{a^2}{b^2c^2}+\frac{b^2}{a^2c^2}+\frac{c^2}{a^2b^2}\right)-\frac{1}{2}\frac{(k+l)^2}{\Delta^2}\frac{f^2}{a^4}-\frac{1}{2}\frac{l^2}{\Delta^2}\frac{f^2}{b^4}-\frac{1}{2}\frac{k^2}{\Delta^2}\frac{f^2}{c^4}.
\end{split}
\end{equation}

Singular orbits of $M_{k,l}^{k+l}$, $M_{k,l}^{l}$ and $M_{k,l}^{k}$ are respectively generated by $\mathbf{2}^{2k+l}\oplus \mathbf{2}^{k+2l}$, $\mathbf{2}^{k-l}\oplus \mathbf{2}^{k+2l}$ and $\mathbf{2}^{k-l}\oplus \mathbf{2}^{2k+l}$. By \cite{eschenburg_initial_2000, reidegeld2011exceptional}, the corresponding initial conditions are
\begin{equation}
\label{eqn: initial condtion}
\begin{split}
& \lim_{t\to 0}(a,b,c,f,\dot{a},\dot{b},\dot{c},\dot{f})=\left(0,a_0,a_0,0,1,0,0,\frac{2\Delta}{k+l}\right),\\
& \lim_{t\to 0}(a,b,c,f,\dot{a},\dot{b},\dot{c},\dot{f})=\left(b_0,0,b_0,0,0,1,0,\frac{2\Delta}{l}\right),\\
& \lim_{t\to 0}(a,b,c,f,\dot{a},\dot{b},\dot{c},\dot{f})=\left(c_0,c_0,0,0,0,0,1,\frac{2\Delta}{k}\right),\\
& a_0, b_0, c_0>0.
\end{split}
\end{equation}

The $\spin(7)$ equations derived in \cite{reidegeld2011exceptional} are
\begin{equation}
\label{eqn: original spin(7) equation}
\begin{split}
\frac{\dot{a}}{a}&=\frac{b}{ac}+\frac{c}{ab}-\frac{a}{bc}-\frac{k+l}{2\Delta}\frac{f}{a^2},\\
\frac{\dot{b}}{b}&=\frac{c}{ab}+\frac{a}{bc}-\frac{b}{ac}+\frac{l}{2\Delta}\frac{f}{b^2},\\
\frac{\dot{c}}{c}&=\frac{a}{bc}+\frac{b}{ac}-\frac{c}{ab}+\frac{k}{2\Delta}\frac{f}{c^2},\\
\frac{\dot{f}}{f}&=\frac{k+l}{2\Delta}\frac{f}{a^2}-\frac{l}{2\Delta}\frac{f}{b^2}-\frac{k}{2\Delta}\frac{f}{c^2}.
\end{split}
\end{equation}
Changing the sign of $f$ in \eqref{eqn: original spin(7) equation} yields the $\spin(7)$ condition with the opposite chirality:
\begin{equation}
\label{eqn: original opposite spin(7) equation}
\begin{split}
\frac{\dot{a}}{a}&=\frac{b}{ac}+\frac{c}{ab}-\frac{a}{bc}+\frac{k+l}{2\Delta}\frac{f}{a^2},\\
\frac{\dot{b}}{b}&=\frac{c}{ab}+\frac{a}{bc}-\frac{b}{ac}-\frac{l}{2\Delta}\frac{f}{b^2},\\
\frac{\dot{c}}{c}&=\frac{a}{bc}+\frac{b}{ac}-\frac{c}{ab}-\frac{k}{2\Delta}\frac{f}{c^2},\\
\frac{\dot{f}}{f}&=-\frac{k+l}{2\Delta}\frac{f}{a^2}+\frac{l}{2\Delta}\frac{f}{b^2}+\frac{k}{2\Delta}\frac{f}{c^2}.
\end{split}
\end{equation}

We follow the change of coordinates in \cite{chi2022spin}, which transforms the Ricci-flat system to a dynamical system on an algebraic surface. Let $L=\frac{1}{2}g_{N_{k,l}}g^{-1}_{N_{k,l}}=\diag\left(\frac{\dot{a}}{a}I_2,\frac{\dot{b}}{b}I_2,\frac{\dot{c}}{c}I_2,\frac{\dot{f}}{f}\right)$ be the shape operator of $N_{k,l}$ in $M_{k,l}^i$. Normalize the orbit space by $d\eta=\trace{L} dt$. Define functions
\begin{equation}
\label{eqn: Xi and Zi}
\begin{split}
&X_1=\frac{\frac{\dot{a}}{a}}{\trace{L}},\quad X_2=\frac{\frac{\dot{b}}{b}}{\trace{L}},\quad X_3=\frac{\frac{\dot{c}}{c}}{\trace{L}},\quad X_4=\frac{\frac{\dot{f}}{f}}{\trace{L}},\\
&Z_1=\frac{\frac{a}{bc}}{\trace{L}},\quad 
Z_2=\frac{\frac{b}{ac}}{\trace{L}},\quad 
Z_3=\frac{\frac{c}{ab}}{\trace{L}},\quad 
Z_4=f\trace{L}.
\end{split}
\end{equation}
Let $'$ denote the derivative with respect to $\eta$. The original dynamical system \eqref{eqn: original Einstein equation} is transformed to 
\begin{equation}
\label{eqn: new Einstein equation}
\begin{bmatrix}
X_1\\
X_2\\
X_3\\
X_4\\
Z_1\\
Z_2\\
Z_3\\
Z_4
\end{bmatrix}'
=\begin{bmatrix}
X_1(G-1)+R_1\\
X_2(G-1)+R_2\\
X_3(G-1)+R_3\\
X_4(G-1)+R_4\\
Z_1(G+X_1-X_2-X_3)\\
Z_2(G+X_2-X_3-X_1)\\
Z_3(G+X_3-X_1-X_2)\\
Z_4(-G+X_4)\\
\end{bmatrix},
\end{equation}
where
\begin{equation}
\begin{split}
&G=2X_1^2+2X_2^2+2X_3^2+X_4^2,\\
&R_1=6Z_2Z_3+Z_1^2-Z_2^2-Z_3^2-\frac{1}{2}\frac{(k+l)^2}{\Delta^2}Z_2^2Z_3^2Z_4^2,\\
&R_2=6Z_1Z_3+Z_2^2-Z_3^2-Z_1^2-\frac{1}{2}\frac{l^2}{\Delta^2}Z_1^2Z_3^2Z_4^2,\\
&R_3=6Z_1Z_2+Z_3^2-Z_1^2-Z_2^2-\frac{1}{2}\frac{k^2}{\Delta^2}Z_1^2Z_2^2Z_4^2,\\
&R_4=\frac{1}{2}\frac{(k+l)^2}{\Delta^2}Z_2^2Z_3^2Z_4^2+\frac{1}{2}\frac{l^2}{\Delta^2}Z_1^2Z_3^2Z_4^2+\frac{1}{2}\frac{k^2}{\Delta^2}Z_1^2Z_2^2Z_4^2.
\end{split}
\end{equation}
The conservation law \eqref{eqn: original conservation} becomes
\begin{equation}
\label{eqn: new conservation}
G-1+2R_1+2R_2+2R_3+R_4=0.
\end{equation}
Since $\left(\frac{1}{\trace(L)}\right)'=\frac{1}{\trace(L)}G$, the quantity $\frac{1}{\trace(L)}$ can be treated as a function of $\eta$ by
$$\frac{1}{\trace(L)}=\exp\left(\int_{\eta^*}^{\eta}G d\tilde{\eta}+C\right).$$
To recover the original coordinates, we simply compute
$$
t=\int_{\eta^*}^\eta \frac{1}{\trace(L)} d\tilde{\eta}=\int_{\eta^*}^\eta \exp\left(\int_{\eta^{**}}^{\eta^*}G d\tilde{\tilde{\eta}}+C\right) d\tilde{\eta}+t_0
$$
and 
$$
a=\frac{1}{\trace(L)}\frac{1}{\sqrt{Z_2Z_3}},\quad b=\frac{1}{\trace(L)}\frac{1}{\sqrt{Z_1Z_3}},\quad  c=\frac{1}{\trace(L)}\frac{1}{\sqrt{Z_1Z_2}},\quad f=\frac{Z_4}{\trace(L)}.
$$

Although our main focus is on constructing $\spin(7)$ metrics, it is more convenient to begin with the full Ricci-flat system, where several key estimates are transparent. From the definition of $X_j$ in \eqref{eqn: Xi and Zi}, one expects the equality
\begin{equation}
\label{eqn_trace=1}
2X_1+2X_2+2X_3+X_4=1
\end{equation}
 be preserved by the new dynamical system. Indeed, since 
\begin{equation}
\label{eqn: H=1}
\begin{split}
(2X_1+2X_2+2X_3+X_4)'&=(2X_1+2X_2+2X_3+X_4)(G-1)+2R_1+2R_2+2R_3+R_4\\
&=(2X_1+2X_2+2X_3+X_4)(G-1)+1-G \quad \text{by \eqref{eqn: new conservation}}\\
&=(2X_1+2X_2+2X_3+X_4-1)(G-1)
\end{split},
\end{equation}
the set $$\{2X_1+2X_2+2X_3+X_4=1\}$$ is invariant. We can also assume each $Z_j$ to be non-negative since the set $\{Z_j = 0\}$ is invariant by \eqref{eqn: new Einstein equation}. A straightforward observation gives the following proposition.
\begin{proposition}
\label{prop: X4 not negative}
The set $\{X_4\geq 0\}$ is invariant.
\end{proposition}
\begin{proof}
We have
\begin{equation}
\left.X_4'\right|_{X_4=0}=R_4=\frac{1}{2}\frac{(k+l)^2}{\Delta^2}Z_2^2Z_3^2Z_4^2+\frac{1}{2}\frac{l^2}{\Delta^2}Z_1^2Z_3^2Z_4^2+\frac{1}{2}\frac{k^2}{\Delta^2}Z_1^2Z_2^2Z_4^2\geq 0.
\end{equation}
If a non-transversal intersection emerges on an integral curve, it is necessary that the integral curve is in the invariant set $\{Z_j=0\}$ for some $j\in\{1,2,3,4\}$. Hence, the derivative $X_4'$ vanishes identically. The possibility of non-transversal intersections is excluded.
\end{proof}

By our discussion above, a cohomogeneity one Ricci-flat metric with a generic $N_{k,l}$ as the principal orbit is represented by an integral curve to \eqref{eqn: new Einstein equation} in the following invariant subset of $\mathbb{R}^8$:
\begin{equation}
\label{eqn: new conservation set}
\begin{split}
\mathcal{C}_{RF}&:=\{ G-1+2R_1+2R_2+2R_3+R_4=0\}\cap \{2X_1+2X_2+2X_3+X_4=1\}\\
&\quad \cap \{Z_1,Z_2,Z_3,Z_4\geq 0\}\cap \{X_4\geq 0\},
\end{split}
\end{equation}
 a $6$-dimensional algebraic surface with boundaries. By \cite[Lemma 5.1]{buzano_family_2015}, we have $\lim\limits_{\eta\to \infty} t=\infty$. Therefore, if an integral curve is defined on $\mathbb{R}$, the corresponding Ricci-flat metric is forward complete.

The $\spin(7)$ equations \eqref{eqn: original spin(7) equation}–\eqref{eqn: original opposite spin(7) equation} are first-order subsystems of the second-order system \eqref{eqn: original Einstein equation}. In the new coordinates, they appear as invariant algebraic surfaces and further reduce the dimension of the phase space. Specifically, equations \eqref{eqn: original spin(7) equation}–\eqref{eqn: original opposite spin(7) equation} become
\begin{equation}
\label{eqn: spin(7)}
\begin{split}
\mathcal{S}_1^\pm&: X_1=Z_2+Z_3-Z_1\mp\frac{k+l}{2\Delta}Z_2Z_3Z_4,\\
\mathcal{S}_2^\pm&: X_2=Z_3+Z_1-Z_2\pm\frac{l}{2\Delta}Z_1Z_3Z_4,\\
\mathcal{S}_3^\pm&: X_3=Z_1+Z_2-Z_3\pm\frac{k}{2\Delta}Z_1Z_2Z_4,\\
\mathcal{S}_4^\pm&: X_4=\pm\frac{k+l}{2\Delta}Z_2Z_3Z_4\mp\frac{l}{2\Delta}Z_1Z_3Z_4\mp\frac{k}{2\Delta}Z_1Z_2Z_4.
\end{split}
\end{equation}
Substituting each $X_j$ in \eqref{eqn: new conservation} using \eqref{eqn: spin(7)}, we obtain
$$
\left(2(Z_1+Z_2+Z_3)\mp\frac{k+l}{2\Delta}Z_2Z_3Z_4\pm\frac{l}{2\Delta}Z_1Z_3Z_4\pm\frac{k}{2\Delta}Z_1Z_2Z_4\right)^2=1.
$$
On the other hand, substituting each $X_j$ in \eqref{eqn_trace=1} yields
\begin{equation}
\label{eqn: conservation Z spin(7) 1}
2(Z_1+Z_2+Z_3)\mp\frac{k+l}{2\Delta}Z_2Z_3Z_4\pm\frac{l}{2\Delta}Z_1Z_3Z_4\pm\frac{k}{2\Delta}Z_1Z_2Z_4=1,
\end{equation}
or equivalently
\begin{equation}
\label{eqn: conservation Z spin(7) 2}
2(Z_1+Z_2+Z_3)-X_4=1.
\end{equation}
By \cite[Proposition 2.2]{chi2022spin}, the sets
\begin{equation}
\label{eqn_spin7 conservation}
\begin{split}
\mathcal{C}^\pm_{\spin(7)}=&\mathcal{C}_{RF}\cap \left(\bigcap\limits_{i=1}^4\mathcal{S}_i^\pm\right)\cap \left\{2(Z_1+Z_2+Z_3)-X_4=1\right\}
\end{split}
\end{equation}
are invariant. A cohomogeneity one $\spin(7)$ metric with a generic $N_{k,l}$ as the principal orbit is represented by an integral curve to \eqref{eqn: new Einstein equation} restricted to $(\mathcal{C}^+_{\spin(7)}\cup \mathcal{C}^-_{\spin(7)}) \subset \mathcal{C}_{RF}$, a 3-dimensional algebraic surface in $\mathbb{R}^8$ with boundaries. 

It is apparent that 
\begin{equation}
\begin{split}
\mathcal{C}_{G_2}=\mathcal{C}^+_{\spin(7)}\cap \mathcal{C}^-_{\spin(7)} \cap \{X_4=0\}\cap \{Z_4=0\}
\end{split}
\end{equation}
is a 2-dimensional invariant subset. System \eqref{eqn: new Einstein equation} restricted to $\mathcal{C}_{G_2}$ is essentially the one for cohomogeneity one $G_2$ metrics with $\mathbb{CP}^2$ as the singular orbit and $SU(3)/T^2$ as the principal orbit. For a forward complete ALC $\spin(7)$ metric, components $a,b$ and $c$ in \eqref{eqn: invariant metric} grow linearly and $f$ converges to a constant as $t\to \infty$. The corresponding integral curve converges to the invariant set $\mathcal{C}_{G_2}$ as $\eta\to\infty$.

\section{Local Existence}
\label{sec: local analysis}
In this section, we study the critical points of the $\spin(7)$ system. These critical points encode the initial conditions \eqref{eqn: initial condtion} and the AC/ALC asymptotics. Linearizations at the initial condition critical points yield the local existence of $\spin(7) $ metrics near the tubular neighbourhood of $\mathbb{CP}^2$, recovering the result of \cite{reidegeld2011exceptional}.

The following is the complete list of critical points of \eqref{eqn: new Einstein equation} on $\mathcal{C}_{\spin(7)}^\pm$. We refer the reader to \cite{chi2022spin} for the complete list of critical points of \eqref{eqn: new Einstein equation} on $\mathcal{C}_{RF}$.
\begin{enumerate}[I]
\item
$P_0^{k+l}:=\left(\frac{1}{3},0,0,\frac{1}{3},0,\frac{1}{3},\frac{1}{3},\frac{6\Delta}{k+l}\right)\in \mathcal{C}_{\spin}^+$;\\
$P_0^{l}:=\left(0,\frac{1}{3},0,\frac{1}{3},\frac{1}{3},0,\frac{1}{3},\frac{6\Delta}{l}\right),\quad P_0^{k}:=\left(0,0,\frac{1}{3},\frac{1}{3},\frac{1}{3},\frac{1}{3},0,\frac{6\Delta}{k}\right)\in \mathcal{C}_{\spin}^-$.

These critical points represent the initial conditions \eqref{eqn: initial condtion} in the new coordinate. Integral curves that emanate from these points represent Ricci-flat metrics that are defined on the tubular neighborhood around $\mathbb{CP}^2$ in $M_{k,l}^{k+l}$, $M_{k,l}^{l}$ and  $M_{k,l}^{k}$, respectively.
\item
\begin{enumerate}
\item
$P_{ALC}:=\left(\frac{1}{6},\frac{1}{6},\frac{1}{6},0,\frac{1}{6},\frac{1}{6},\frac{1}{6},0\right)\in\mathcal{C}_{G_2}$
\end{enumerate}

If an integral curve converges to $P_{ALC}$, the corresponding metric has ALC asymptotics, with its end modeled on an 
$\mathbb{S}^1$-bundle over the cone on the homogeneous nearly K\"ahler metric $SU(3)/T^2$.

\item
$Q_1^{k+l}:=\left(\frac{1}{2},-\frac{1}{2},-\frac{1}{2},0,\frac{\sqrt{5}}{2},0,0,0\right),\quad  Q_1^l:=\left(-\frac{1}{2},\frac{1}{2},-\frac{1}{2},0,0,\frac{\sqrt{5}}{2},0,0\right),\quad  Q_1^k:=\left(-\frac{1}{2},-\frac{1}{2},\frac{1}{2},0,0,0,\frac{\sqrt{5}}{2},0\right)\in \mathcal{C}_{G_2}.$

These critical points are sources in the subsystem on $\mathcal{C}_{G_2}$. Singular $G_2$ metrics in \cite{cleyton_cohomogeneity-one_2002} are represented by integral curves that emanate from these points.
\item
$Q_0^{k+l}:=\left(\frac{1}{2},0,0,0,0,\frac{1}{4},\frac{1}{4},0\right),\quad  Q_0^l:=\left(0,\frac{1}{2},0,0,\frac{1}{4},0,\frac{1}{4},0\right),\quad Q_0^k:=\left(0,0,\frac{1}{2},0,\frac{1}{4},\frac{1}{4},0,0\right)\in \mathcal{C}_{G_2}$.

These critical points are saddles in the subsystem on $\mathcal{C}_{G_2}$. The smooth $G_2$ metrics that collapse to $\mathbb{CP}^2$ in \cite{bryant_construction_1989, gibbons_einstein_1990} are represented by integral curves that emanate from these points.

\item
$\left(\frac{1}{7},\frac{1}{7},\frac{1}{7},\frac{1}{7},z_1,z_2,z_3,z_4\right)$, where $R_i(z_1,z_2,z_3,z_4)=\frac{6}{49}$ for each $i$.

There are exactly two critical points of this type, corresponding to the homogeneous Einstein metrics on  $N_{k,l}$ \cite{wang_some_1982, castellani1984n, page1984new, kowalski_homogeneous_1993}. Viewed as trivial integral curves of \eqref{eqn: new Einstein equation}, they represent Ricci-flat metric cones over these homogeneous Einstein metrics, which are in fact $\spin(7)$ \cite{bar1993real}. Below we recover the existence of these critical points and show that the associated $\spin(7)$ cones have opposite chirality.
\end{enumerate}

\begin{proposition}
\label{prop: spin7_cone_chirality}
Each of $\mathcal{C}_{\spin(7)}^\pm$ admits exactly one critical point of Type~V, respectively denoted as $P_{AC}^\pm$. 
\end{proposition}
\begin{proof}
By the aforementioned works, there are exactly two critical points of Type V in $\mathcal{C}_{\spin(7)}^+\cup \mathcal{C}_{\spin(7)}^-$. We show that each of $\mathcal{C}_{\spin(7)}^\pm$ contains one.

By \eqref{eqn: conservation Z spin(7) 2}, Type V critical points on $\mathcal{C}_{\spin(7)}^\pm$ are on the hypersurface $\{Z_1+Z_2+Z_3=\frac{4}{7}\}$. Substituting $X_j=\frac{1}{7}=\frac{1}{4}(Z_1+Z_2+Z_3)$ in $\mathcal{S}_j^\pm$ for $j\in \{1,2,3\}$, we have 
\begin{equation}
\begin{split}
&\frac{1}{4Z_2Z_3}(3Z_2+3Z_3-5Z_1)\mp \frac{k+l}{2\Delta}Z_4=0,\\
&\frac{1}{4Z_3Z_1}(3Z_3+3Z_1-5Z_2)\pm \frac{l}{2\Delta}Z_4=0,\\
&\frac{1}{4Z_1Z_2}(3Z_1+3Z_2-5Z_3)\pm\frac{k}{2\Delta}Z_4=0.
\end{split}
\end{equation}
Summing the above three equations, we have 
$$5(Z_1^2+Z_2^2+Z_3^2)=6(Z_2Z_3+Z_3Z_1+Z_1Z_2),$$ while the last two equations above yield
$$
kZ_2(3Z_3+3Z_1-5Z_2)=lZ_3(3Z_1+3Z_2-5Z_3).
$$
Therefore, Type V critical points are realized as intersections between the the circle 
$$\left\{Z_1+Z_2+Z_3=\frac{4}{7}\right\}\cap \{5(Z_1^2+Z_2^2+Z_3^2)=6(Z_2Z_3+Z_3Z_1+Z_1Z_2)\}.$$
and the hyperbola
$$\left\{Z_1+Z_2+Z_3=\frac{4}{7}\right\}\cap \{kZ_2(3Z_3+3Z_1-5Z_2)=lZ_3(3Z_1+3Z_2-5Z_3)\}.$$
Consider $\alpha=\frac{Z_2}{Z_1}$, $\beta=\frac{Z_3}{Z_1}$, it suffices to solve the equations $L_1=L_2=0$, where
\begin{equation}
\begin{split}
L_1&=5(1+\alpha^2+\beta^2)-(6\alpha\beta+6\beta+6\alpha)=(\alpha+\beta-3)^2+4(\alpha-\beta)^2-4\\
L_2&=k\alpha(3+3\beta-5\alpha)-l\beta(3+3\alpha-5\beta).
\end{split}
\end{equation}
The level curve $L_1=0$ is an ellipse in the $(\alpha,\beta)$-space that passes through
$$
\left(\frac{2}{5},1\right),\quad \left(1,\frac{2}{5}\right),\quad \left(2,\frac{13}{5}\right),\quad \left(\frac{13}{5},2\right),
$$
while the level curve $L_2=0$ is a hyperbola. Since 
$$
L_2\left(\frac{2}{5},1\right)=k\frac{8}{5}+l\frac{4}{5},\quad L_2\left(1,\frac{2}{5}\right)=-k\frac{4}{5}-l\frac{8}{5},$$
$$L_2\left(2,\frac{13}{5}\right)=k\frac{8}{5}+l\frac{52}{5},\quad L_2\left(\frac{13}{5},2\right)=-k\frac{52}{5}-l\frac{8}{5},
$$
there exists an intersection point on the elliptical arc
$$\{L_1=0\}\cap \left\{\alpha+\beta<\frac{7}{5}\right\}=\{L_1=0\}\cap \left\{\alpha+\beta<\frac{7}{5}\right\}\cap \left\{-\frac{3}{5}<\alpha-\beta<\frac{3}{5}\right\},$$
another on 
$$\{L_1=0\}\cap \left\{\alpha+\beta>\frac{23}{5}\right\}=\{L_1=0\}\cap \left\{\alpha+\beta>\frac{23}{5}\right\}\cap \left\{-\frac{3}{5}<\alpha-\beta<\frac{3}{5}\right\}.$$

For a Type V critical point, we have 
$$\frac{1}{7}=X_4=\pm\left(\frac{k+l}{2\Delta}Z_2Z_3-\frac{l}{2\Delta}Z_1Z_3-\frac{k}{2\Delta}Z_1Z_2\right)Z_4.$$ For the first intersection point, we have $\alpha, \beta < 1$,  equivalently $Z_2,Z_3 < Z_1$. This makes the expression in parentheses negative and thus the point lies in $\mathcal{C}_{\spin(7)}^-$ in order for $Z_4>0$. For the second intersection point, we have $\alpha,\beta>2$, equivalently $Z_2,Z_3>2Z_1$. The expression in parentheses is positive, so the intersection point must lie in $\mathcal{C}_{\spin(7)}^+$ for $Z_4>0$.
\end{proof}

Each Type I critical point $P_0^i$ is hyperbolic with a single positive eigenvalue $\frac{2}{3}$, and it has two unstable eigenvectors $v_1$ and $v_2$ that are tangent to $\mathcal{C}_{\spin(7)}^\pm$. Below we list $v_1$ and $v_2$ for each $P_0^i$.
$$
\begin{array}{c|c|c|c}
&P_0^{k+l}&P_0^{l}&P_0^{k}\\
\hline
&&&\\
v_1& \begin{bmatrix}
2\\
0\\
0\\
-4\\
0\\
-1\\
-1\\
-36\frac{\Delta}{k+l}
\end{bmatrix}&\begin{bmatrix}
0\\
2\\
0\\
-4\\
-1\\
0\\
-1\\
-36\frac{\Delta}{l}
\end{bmatrix}&\begin{bmatrix}
0\\
0\\
2\\
-4\\
-1\\
-1\\
0\\
-36\frac{\Delta}{k}
\end{bmatrix}\\
&&&\\
v_2& \begin{bmatrix}
-3(k+l)\\
4k+5l\\
5k+4l\\
-12(k+l)\\
3(k+l)\\
-5k-4l\\
-4k-5l\\
0
\end{bmatrix}&\begin{bmatrix}
5l+k\\
-3l\\
4l-k\\
-12l\\
-4l+k\\
3l\\
-5l-k\\
0
\end{bmatrix}&\begin{bmatrix}
5k+l\\
4k-l\\
-3k\\
-12k\\
-4k+l\\
-5k-l\\
3k\\
0
\end{bmatrix}\\[1ex]
\end{array}.
$$

By the Hartman--Grobman Theorem, there is a one-to-one correspondence between linearized solutions and the corresponding integral curves emanating from $P_0^i$. By the unstable version of \cite[Theorem 4.5]{coddington_theory_1955}, it is unambiguous to denote the integral curve emanating from $P_0^{i}$ by $\gamma^{i}_{\theta}$, where
\begin{equation}
\label{eqn_linearized solution}
\gamma^{i}_{\theta}=P_0^{i}+s_1 e^{\frac{2\eta}{3}}v_1+s_2 e^{\frac{2\eta}{3}}v_2+O(e^{\left(\frac{2}{3}+\epsilon\right)\eta}),\quad  (s_1,s_2)=(\cos(\theta),\sin(\theta)).
\end{equation}
The normalization $s_1^2+s_2^2=1$ removes the scaling redundancy, and we set $\theta\in [0,\pi]$ so that each $Z_j$ is non-negative. We hence obtain 
$$\{\gamma^{i}_{\theta}\mid \theta\in[0,\pi]\},$$
a continuous one-parameter family of $\spin(7)$ metrics defined on a tubular neighbourhood of $\mathbb{CP}^2$ in $M_{k,l}^i$. 
All integral curves in the interior represent non-degenerate metrics, since each $Z_j>0$.
The two integral curves at the boundaries $\theta=0,\pi$ represent degenerate metrics, as some $Z_j$ vanishes identically. Specifically, for each $i$, integral curves $\gamma^i_0$ and $\gamma^i_{\pi}$ lie on the following flow-invariant algebraic curves $\mathcal{W}_i$. 
\begin{equation}
\label{eqn_W curve}
\begin{split}
\mathcal{W}_{k+l}&:=\mathcal{C}_{\spin(7)}^+\cap \{X_2=X_3=0\}\cap \{X_1=1-2Z_2\}\cap \{X_4=4Z_2-1\}\\
&\quad \cap \{Z_2=Z_3\}\cap \{Z_1=0\}\cap \left\{Z_4=\frac{2\Delta}{k+l} \frac{4Z_2-1}{Z_2^2}\right\},\\
\mathcal{W}_l&:=\mathcal{C}_{\spin(7)}^-\cap \{X_1=X_3=0\}\cap \{X_2=1-2Z_3\}\cap \{X_4=4Z_3-1\}\\
&\quad \cap \{Z_3=Z_1\}\cap \{Z_2=0\}\cap \left\{Z_4=\frac{2\Delta}{l} \frac{4Z_3-1}{Z_3^2}\right\},\\
\mathcal{W}_k&:=\mathcal{C}_{\spin(7)}^-\cap \{X_1=X_2=0\}\cap \{X_3=1-2Z_1\}\cap \{X_4=4Z_1-1\}\\
&\quad \cap \{Z_1=Z_2\}\cap \{Z_3=0\}\cap \left\{Z_4=\frac{2\Delta}{k} \frac{4Z_1-1}{Z_1^2}\right\}.
\end{split}
\end{equation}
In particular, we have 
\begin{equation}
\label{eqn_gamma_0^i}
\begin{split}
\gamma_0^{k+l}&=\mathcal{W}_{k+l}\cap \left\{\frac{1}{4}<Z_2=Z_3<\frac{1}{3}\right\}\\
\gamma_0^{l}&=\mathcal{W}_{l}\cap \left\{\frac{1}{4}<Z_3=Z_1<\frac{1}{3}\right\}\\
\gamma_0^{k}&=\mathcal{W}_{k}\cap \left\{\frac{1}{4}<Z_1=Z_2<\frac{1}{3}\right\}\\
\end{split}
\end{equation}
Furthermore, the inequality $Z_4<\frac{6\Delta}{i}$ holds along each $\gamma_0^i$ and the integral curve joins $P_0^i$ and $Q_0^i$. On the other hand, we have $Z_4>\frac{6\Delta}{i}$ along each $\gamma_\pi^i$.

Geometrically, the parameter $\theta$ governs the initial differences among principal curvatures. For example, for $\gamma_\theta^{k+l}$, we have 
\begin{equation}
\begin{split}
\cot(\theta)&=\frac{s_1}{s_2}=\lim\limits_{\eta\to-\infty}\frac{k+l}{2}\frac{X_1-X_2-X_3-X_4}{Z_1}=\lim\limits_{t\to 0}\frac{k+l}{2}\frac{bc}{a}\left(\frac{\dot{a}}{a}-\frac{\dot{b}}{b}-\frac{\dot{c}}{c}-\frac{\dot{f}}{f}\right)\\
&=-\frac{3(k+l)}{16\Delta}\left(12\Delta+q\right),
\end{split}
\end{equation}
where $q$ is the free third-order parameter appearing in the power series \cite[(7.1)]{reidegeld2011exceptional}.

\section{Invariant sets for ALC metrics}
\label{sec: invariant set 1}
The compact invariant sets in \cite{chi2022spin} are constructed using the quantity $Z_4=f\mathrm{tr}(L)$. If $k>l$,
the condition
$$
f\mathrm{tr}(L)\leq \lim_{t\to 0} \left(f\mathrm{tr}(L)\right)(\gamma^i_\theta)
$$
defines a compact invariant set in the $(X_j,Z_j)$-space for $i\in \{k+l,k\}$, which helps prove the forward completeness for metrics on $M_{k,l}^{k+l}$ and $M_{k,l}^{k}$. The set fails to be compact if $i=l$. This is the essential limitation of the old construction.

In this section, we introduce a new inequality,
\begin{equation}
\label{eqn_new inequality}
(k+l)\frac{f}{a}+l\frac{f}{b}+k\frac{f}{c}\leq 2\Delta.
\end{equation}
This condition is more geometrically natural, as it compares the $\mathbb{S}^1$-fiber in $N_{k,l}$ with the other metric components. The inequality admits dihedral symmetry among pairs $(a,k+l)$, $(b,l)$, and $(c,k)$. The associated invariant sets bound $Z_1$, $Z_2$, $Z_3$ simultaneously. This helps prove that a $\gamma^i_\theta$ with a sufficiently small $\theta\geq 0$ remains in a compact subset for 
\emph{all three cases} $i\in \{k+l,l,k\}$. In Section~\ref{subsec_+} we show that \eqref{eqn_new inequality} defines an invariant set inside $\mathcal{C}_{\spin(7)}^{+}$ for $\gamma^{k+l}_\theta$.  
In Section~\ref{subsec_-} we prove the analogous statement in $\mathcal{C}_{\spin(7)}^{-}$ for $\gamma^{l}_\theta$ and $\gamma^{k}_\theta$.

\subsection{$\mathcal{C}^+_{\spin(7)}$}
\label{subsec_+}
Define 
\begin{equation}
\label{eqn_D+}
\begin{split}
\mathcal{D}^+&:=\mathcal{C}^+_{\spin(7)}\cap \{Z_1\leq  Z_2\}\cap \{Z_1\leq  Z_3\}\\
&\quad \cap \left\{\left(k\sqrt{Z_1Z_2}+l\sqrt{Z_1Z_3}+(k+l)\sqrt{Z_2Z_3}\right)Z_4\leq 2\Delta\right\}.
\end{split}
\end{equation}
\begin{lemma}
\label{lem: invariant+}
The set $\mathcal{D}^+$ is invariant.
\end{lemma}
\begin{proof}
By \eqref{eqn: new Einstein equation} and \eqref{eqn: spin(7)}, we have 
\begin{equation}
\label{eqn: Z1<Z2}
\begin{split}
\left.(Z_2-Z_1)'\right|_{Z_2-Z_1=0}&=2Z_1(X_2-X_1)\\
&=2Z_1\left(\frac{l}{2\Delta}Z_1Z_3Z_4+\frac{k+l}{2\Delta}Z_2Z_3Z_4\right)\\
&\geq 0
\end{split},
\end{equation}
\begin{equation}
\label{eqn: Z1<Z3}
\begin{split}
\left.(Z_3-Z_1)'\right|_{Z_3-Z_1=0}&=2Z_1(X_3-X_1)\\
&=2Z_1\left(\frac{k}{2\Delta}Z_1Z_2Z_4+\frac{k+l}{2\Delta}Z_2Z_3Z_4\right)\\
&\geq 0
\end{split}.
\end{equation}
It suffices to show that an integral curve in $\mathcal{D}^+$ does not escape through the boundary $\mathcal{D}^+\cap \{\left(k\sqrt{Z_1Z_2}+l\sqrt{Z_1Z_3}+(k+l)\sqrt{Z_2Z_3}\right)Z_4= 2\Delta\}$.

We have 
\begin{equation}
\label{eqn: technical computation+ 1-1}
\begin{split}
&\left.\left(\left(k\sqrt{Z_1Z_2}+l\sqrt{Z_1Z_3}+(k+l)\sqrt{Z_2Z_3}\right)Z_4\right)'\right|_{\left(k\sqrt{Z_1Z_2}+l\sqrt{Z_1Z_3}+(k+l)\sqrt{Z_2Z_3}\right)Z_4=2\Delta}
\\
&=k\sqrt{Z_1Z_2}Z_4(X_4-X_3)+l\sqrt{Z_1Z_3}Z_4(X_4-X_2)+(k+l)\sqrt{Z_2Z_3}Z_4(X_4-X_1)\\
&=Z_4\left(k\sqrt{Z_1Z_2}+l\sqrt{Z_1Z_3}+(k+l)\sqrt{Z_2Z_3}\right)X_4\\
&\quad +Z_4(k\sqrt{Z_1Z_2}(-X_3)+l\sqrt{Z_1Z_3}(-X_2)+(k+l)\sqrt{Z_2Z_3}(-X_1))\\
&=Z_4\left(k\sqrt{Z_1Z_2}+l\sqrt{Z_1Z_3}+(k+l)\sqrt{Z_2Z_3}\right)\left(\frac{k+l}{2\Delta}Z_2Z_3Z_4-\frac{l}{2\Delta}Z_1Z_3Z_4-\frac{k}{2\Delta}Z_1Z_2Z_4\right) \\
&\quad +k\sqrt{Z_1Z_2}Z_4\left(-\frac{k}{2\Delta}Z_1Z_2Z_4+Z_3-Z_1-Z_2\right)\\
&\quad +l\sqrt{Z_1Z_3}Z_4\left(-\frac{l}{2\Delta}Z_1Z_3Z_4+Z_2-Z_3-Z_1\right)\\
&\quad +(k+l)\sqrt{Z_2Z_3}Z_4\left(\frac{k+l}{2\Delta}Z_2Z_3Z_4+Z_1-Z_2-Z_3\right).
\end{split}
\end{equation}
Let
\begin{equation}
\label{eqn: alpha beta gamma in techinal}
Z_1Z_4=\zeta,\quad \sqrt{\frac{Z_2}{Z_1}}=\alpha, \quad  \sqrt{\frac{Z_3}{Z_1}}=\beta.
\end{equation}

The above computation becomes
\begin{equation}
\label{eqn: technical computation+ 1-2}
\begin{split}
&\left.\left(\left(k\sqrt{Z_1Z_2}+l\sqrt{Z_1Z_3}+(k+l)\sqrt{Z_2Z_3}\right)Z_4\right)'\right|_{\left(k\sqrt{Z_1Z_2}+l\sqrt{Z_1Z_3}+(k+l)\sqrt{Z_2Z_3}\right)Z_4=2\Delta}
\\
&=Z_1\zeta^2(k\alpha+l\beta+(k+l)\alpha\beta)\left(\frac{k+l}{2\Delta}\alpha^2\beta^2-\frac{l}{2\Delta}\beta^2-\frac{k}{2\Delta}\alpha^2\right)\\
&\quad +Z_1\zeta^2\left(\frac{(k+l)^2}{2\Delta}\alpha^3\beta^3-\frac{l^2}{2\Delta}\beta^3-\frac{k^2}{2\Delta}\alpha^3\right)\\
&\quad +Z_1\zeta(k\alpha(\beta^2-1-\alpha^2)+l\beta(\alpha^2-1-\beta^2)+(k+l)\alpha\beta(1-\alpha^2-\beta^2)).
\end{split}
\end{equation}
With $\left(k\sqrt{Z_1Z_2}+l\sqrt{Z_1Z_3}+(k+l)\sqrt{Z_2Z_3}\right)Z_4=2\Delta$, we have 
\begin{equation}
\label{eqn_homogenized boundary}
\zeta=\frac{2\Delta}{k\alpha+l\beta+(k+l)\alpha\beta}.
\end{equation}
The equation \eqref{eqn: technical computation+ 1-2} becomes
\begin{equation}
\label{eqn: technical computation+ 1-3}
\begin{split}
&\left.\left(\left(k\sqrt{Z_1Z_2}+l\sqrt{Z_1Z_3}+(k+l)\sqrt{Z_2Z_3}\right)Z_4\right)'\right|_{\left(k\sqrt{Z_1Z_2}+l\sqrt{Z_1Z_3}+(k+l)\sqrt{Z_2Z_3}\right)Z_4=2\Delta}
\\
&=\frac{Z_1\zeta}{k\alpha+l\beta+(k+l)\alpha\beta}(-\beta^2\Xi_0 l^2-\alpha\beta\Xi_1kl-\alpha^2\Xi_2k^2),
\end{split}
\end{equation}
where
\begin{equation}
\begin{split}
\Xi_0&=(\alpha+1)^2 \beta^2+(1-\alpha)(2\alpha^2+3\alpha+2) \beta+( \alpha^2-1)^2,\\
\Xi_1&=2\alpha\beta(\alpha-\beta)^2+(\alpha+\beta)(2\alpha^2-3\alpha\beta+2\beta^2)+(\alpha-\beta)^2+\alpha+\beta+2,\\
\Xi_2&=(\beta+1)^2\alpha^2+(1-\beta)(2\beta^2+3\beta+2)\alpha+(\beta^2-1)^2.\\
\end{split}
\end{equation}

Since $Z_2,Z_3\geq Z_1$ in $\mathcal{D}^+$, we consider each $\Xi_j$ for $(\alpha,\beta)\in [1,\infty)\times [1,\infty)$. Since the discriminant of $\Xi_0$ (as a quadratic function of $\beta$) is 
$$\delta_\beta (\Xi_0)=-\alpha(\alpha-1)^2(4\alpha^2+7\alpha+4)\leq 0,$$
 the function $\Xi_0\geq 0$ and vanishes only if $(\alpha,\beta)=(1,0)$. By the transformation $\alpha\leftrightarrow \beta$, the function $\Xi_2$ is also non-negative and vanishes only if $(\alpha,\beta)=(0,1)$. The function $\Xi_1$ is apparently positive. Therefore, the derivative \eqref{eqn: technical computation+ 1-1} is negative if $Z_1\neq 0$. 

If an integral curve leaves $\mathcal{D}^+$ non-transversally through $$\{\left(k\sqrt{Z_1Z_2}+l\sqrt{Z_1Z_3}+(k+l)\sqrt{Z_2Z_3}\right)Z_4=2\Delta\}\cap \{Z_1=0\},$$ we have
\begin{equation}
\label{eqn: technical computation 1-3}
\begin{split}
&\left.\left(\left(k\sqrt{Z_1Z_2}+l\sqrt{Z_1Z_3}+(k+l)\sqrt{Z_2Z_3}\right)Z_4\right)'\right|_{\{\left(k\sqrt{Z_1Z_2}+l\sqrt{Z_1Z_3}+(k+l)\sqrt{Z_2Z_3}\right)Z_4=2\Delta\}\cap \{Z_1=0\}}
\\
&=(k+l)\sqrt{Z_2Z_3}Z_4(X_4-X_1)\\
&=2\Delta\left(\frac{k+l}{\Delta}Z_2Z_3Z_4-Z_2-Z_3\right)\\
&=2\Delta\left(2\sqrt{Z_2Z_3}-Z_2-Z_3\right)\\
&\leq 0.
\end{split}
\end{equation}
Therefore, the non-transversal intersection satisfies 
$$(k+l)\sqrt{Z_2Z_3}Z_4=2\Delta,\quad Z_1=0,\quad Z_2=Z_3.$$ By \eqref{eqn: spin(7)} and \eqref{eqn: conservation Z spin(7) 1}, the intersection point is $P_0^{k+l}$, a contradiction. Hence, the set $\mathcal{D}^+$ is invariant.
\end{proof}

\subsection{$\mathcal{C}^-_{\spin(7)}$}
\label{subsec_-}
Define 
\begin{equation}
\label{eqn_D-}
\begin{split}
\mathcal{D}^-&:=\mathcal{C}^-_{\spin(7)}\cap \{Z_1\geq Z_2\}\cap \{Z_1\geq Z_3\}\\
&\quad \cap \left\{\left(k\sqrt{Z_1Z_2}+l\sqrt{Z_1Z_3}+(k+l)\sqrt{Z_2Z_3}\right)Z_4\leq 2\Delta\right\}.
\end{split}
\end{equation}

\begin{lemma}
\label{lem: invariant-}
The set $\mathcal{D}^-$ is invariant.
\end{lemma}
\begin{proof}
By \eqref{eqn: new Einstein equation} and \eqref{eqn: spin(7)}, we have 
\begin{equation}
\label{eqn: Z1>Z2}
\begin{split}
\left.(Z_1-Z_2)'\right|_{Z_1-Z_2=0}&=2Z_1(X_1-X_2)\\
&=2Z_1\left(\frac{k+l}{2\Delta}Z_2Z_3Z_4+\frac{l}{2\Delta}Z_1Z_3Z_4\right)\\
&\geq 0
\end{split},
\end{equation}
\begin{equation}
\label{eqn: Z1>Z3}
\begin{split}
\left.(Z_1-Z_3)'\right|_{Z_1-Z_3=0}&=2Z_1(X_1-X_3)\\
&=2Z_1\left(\frac{k+l}{2\Delta}Z_2Z_3Z_4+\frac{k}{2\Delta}Z_1Z_2Z_4\right)\\
&\geq 0
\end{split}.
\end{equation}
For $\mathcal{D}^-\cap \{\left(k\sqrt{Z_1Z_2}+l\sqrt{Z_1Z_3}+(k+l)\sqrt{Z_2Z_3}\right)Z_4= 2\Delta\}$, we have 
\begin{equation}
\label{eqn: technical computation 2-1}
\begin{split}
&\left.\left(\left(k\sqrt{Z_1Z_2}+l\sqrt{Z_1Z_3}+(k+l)\sqrt{Z_2Z_3}\right)Z_4\right)'\right|_{\left(k\sqrt{Z_1Z_2}+l\sqrt{Z_1Z_3}+(k+l)\sqrt{Z_2Z_3}\right)Z_4=2\Delta}
\\
&=k\sqrt{Z_1Z_2}Z_4(X_4-X_3)+l\sqrt{Z_1Z_3}Z_4(X_4-X_2)+(k+l)\sqrt{Z_2Z_3}Z_4(X_4-X_1)\\
&=Z_4(k\sqrt{Z_1Z_2}+l\sqrt{Z_1Z_3}+(k+l)\sqrt{Z_2Z_3})\left(-\frac{k+l}{2\Delta}Z_2Z_3Z_4+\frac{l}{2\Delta}Z_1Z_3Z_4+\frac{k}{2\Delta}Z_1Z_2Z_4\right) \\
&\quad +k\sqrt{Z_1Z_2}Z_4\left(\frac{k}{2\Delta}Z_1Z_2Z_4+Z_3-Z_1-Z_2\right)\\
&\quad +l\sqrt{Z_1Z_3}Z_4\left(\frac{l}{2\Delta}Z_1Z_3Z_4+Z_2-Z_3-Z_1\right)\\
&\quad +(k+l)\sqrt{Z_2Z_3}Z_4\left(-\frac{k+l}{2\Delta}Z_2Z_3Z_4+Z_1-Z_2-Z_3\right).\\
&=Z_1\zeta^2(k\alpha+l\beta+(k+l)\alpha\beta)\left(-\frac{k+l}{2\Delta}\alpha^2\beta^2+\frac{l}{2\Delta}\beta^2+\frac{k}{2\Delta}\alpha^2\right)\\
&\quad +Z_1\zeta^2\left(-\frac{(k+l)^2}{2\Delta}\alpha^3\beta^3+\frac{l^2}{2\Delta}\beta^3+\frac{k^2}{2\Delta}\alpha^3\right)\\
&\quad +Z_1\zeta(k\alpha(\beta^2-1-\alpha^2)+l\beta(\alpha^2-1-\beta^2)+(k+l)\alpha\beta(1-\alpha^2-\beta^2))
\end{split}
\end{equation}
where $(\zeta,\alpha,\beta)$ are as in \eqref{eqn: alpha beta gamma in techinal}. Again by \eqref{eqn_homogenized boundary}, the equation \eqref{eqn: technical computation 2-1} becomes
\begin{equation}
\label{eqn: technical computation 2-2}
\begin{split}
&\left.\left(\left(k\sqrt{Z_1Z_2}+l\sqrt{Z_1Z_3}+(k+l)\sqrt{Z_2Z_3}\right)Z_4\right)'\right|_{\left(k\sqrt{Z_1Z_2}+l\sqrt{Z_1Z_3}+(k+l)\sqrt{Z_2Z_3}\right)Z_4=2\Delta}
\\
&=\frac{Z_1\zeta}{k\alpha+l\beta+(k+l)\alpha\beta}(-\beta^2\Theta_0 l^2-\alpha\beta\Theta_1kl-\alpha^2\Theta_2k^2),
\end{split}
\end{equation}
where
\begin{equation}
\begin{split}
\Theta_0&=(\alpha+1)^2 \beta^2+(\alpha-1)(2\alpha^2+3\alpha+2) \beta+( \alpha^2-1)^2\\
\Theta_1&=2\alpha\beta (\alpha+\beta)^2 + (\alpha+\beta)(2\alpha^2-\alpha\beta+2\beta^2) - (\alpha+\beta)^2 - ( \alpha +  \beta) + 2\\
\Theta_2&=(\beta+1)^2\alpha^2+(\beta-1)(2\beta^2+3\beta+2)\alpha+(\beta^2-1)^2\\
\end{split}
\end{equation}

Since $Z_1\geq Z_2,Z_3$ in $\mathcal{D}^-$, we consider each $\Theta_j$ for $(\alpha,\beta)\in [0,1]\times [0,1]$. Note that
$$\delta_\beta (\Theta_0)=\delta_\beta (\Xi_0)\leq 0.$$ 
Thus the function $\Theta_0\geq 0$ and vanishes only if $(\alpha,\beta)\in \{(1,0),(0,1)\}$. By the transformation $\alpha\leftrightarrow \beta$, the function $\Theta_2$ has the same property. We show that $\Theta_1> 0$ in the following. 

Since
$$
\Theta_1(0,\beta)=(\beta+1)(2\beta^2-3\beta+2)>0,\quad \Theta_1(\alpha,0)=(\alpha+1)(2\alpha^2-3\alpha+2)>0,
$$
it suffices to show that the minimum of $\Theta_1$ in $[0,+\infty)\times [0,+\infty)$ is positive. Since 
$$
\frac{\partial \Theta_1}{\partial \alpha}-\frac{\partial \Theta_1}{\partial \beta}=(\beta-\alpha)(\beta+\alpha)(2\beta+2\alpha-5),
$$
the minimum of $\Theta_1$ in the interior (if it exists) satisfies $2\beta+2\alpha-5$ or $\alpha=\beta$.
Since 
$
\Theta_1\left(\alpha,\frac{5-2\alpha}{2}\right)=\frac{49}{2}>0,
$ and 
$$
\Theta_1(\alpha,\alpha)=8\alpha^4 + 6\alpha^3 - 4\alpha^2 - 2\alpha + 2=3\alpha^4+2\alpha^3+(\alpha^4-\alpha^2+1)+(2\alpha^2+\alpha-1)^2>0,
$$
the function $\Theta_1$ is positive on $[0,+\infty)\times [0,+\infty)$. Since $Z_1\geq Z_2,Z_3$ in $\mathcal{D}^-$, the vanishing of $Z_1$ forces $Z_2$ and $Z_3$ to vanish. Therefore, the variable $Z_1$ does not vanish and one of $\alpha,\beta$ must be positive on the boundary 
$$\mathcal{D}^-\cap \{\left(k\sqrt{Z_1Z_2}+l\sqrt{Z_1Z_3}+(k+l)\sqrt{Z_2Z_3}\right)Z_4=2\Delta\}.$$
Therefore, the polynomial $-\beta^2\Theta_0 l^2-\alpha\beta\Theta_1kl-\alpha^2\Theta_2k^2\leq 0$ and only vanishes at $(1,0)$ and $(0,1)$. Equivalently, the derivative \eqref{eqn: technical computation 2-2} is non-positive and only vanishes at critical points $P_0^{k}$ and $P_0^{l}$. The proof is complete.
\end{proof}

\begin{proposition}
\label{prop: bounded Z1Z2Z3}
Each $Z_j$ with $j\in\{1,2,3\}$ is bounded in $\mathcal{D}^\pm$. 
\end{proposition}
\begin{proof}
If $k\sqrt{Z_1Z_2}+l\sqrt{Z_1Z_3}+(k+l)\sqrt{Z_2Z_3}=0$, at least two of $Z_1$, $Z_2$, $Z_3$ vanish. Hence, the variable $X_4$ vanishes by \eqref{eqn: spin(7)}. The equation \eqref{eqn: conservation Z spin(7) 2} implies
$$2(Z_1+Z_2+Z_3)=1+X_4=1.
$$
Therefore, each $Z_j$ with $j\in\{1,2,3\}$ is bounded in this case.

Now consider the case $k\sqrt{Z_1Z_2}+l\sqrt{Z_1Z_3}+(k+l)\sqrt{Z_2Z_3}\neq 0$.
By \eqref{eqn: conservation Z spin(7) 1} and the definition of $\mathcal{D}^+$, we have
\begin{equation}
\begin{split}
2(Z_1+Z_2+Z_3)&=1+\left(\frac{l}{2\Delta}(Z_2-Z_1)Z_3+\frac{k}{2\Delta}(Z_3-Z_1)Z_2\right)Z_4\\
&\leq 1+\frac{l(Z_2-Z_1)Z_3+k(Z_3-Z_1)Z_2}{k\sqrt{Z_1Z_2}+l\sqrt{Z_1Z_3}+(k+l)\sqrt{Z_2Z_3}}\\
&\leq 1+\frac{(k+l)Z_2Z_3}{(k+l)\sqrt{Z_2Z_3}}\\
&= 1+\sqrt{Z_2Z_3}.
\end{split}
\end{equation}
Therefore, we have $2Z_1+\frac{1}{2}(\sqrt{Z_2}-\sqrt{Z_3})^2+\frac{3}{2}(Z_2+Z_3)\leq 1$. Thus $Z_1$, $Z_2$, $Z_3$ are all bounded. 

On the other hand, by \eqref{eqn: conservation Z spin(7) 1} and the definition of $\mathcal{D}^-$, we have
\begin{equation}
\begin{split}
2(Z_1+Z_2+Z_3)&=1+\left(\frac{l}{2\Delta}(Z_1-Z_2)Z_3+\frac{k}{2\Delta}(Z_1-Z_3)Z_2\right)Z_4\\
&\leq 1+\frac{l(Z_1-Z_2)Z_3+k(Z_1-Z_3)Z_2}{k\sqrt{Z_1Z_2}+l\sqrt{Z_1Z_3}+(k+l)\sqrt{Z_2Z_3}}\\
&\leq 1+\frac{lZ_1Z_3+kZ_1Z_2}{k\sqrt{Z_1Z_2}+l\sqrt{Z_1Z_3}}\\
&\leq 1+\frac{lZ_1Z_3+kZ_1Z_2}{kZ_2+lZ_3}\quad \text{since $Z_1\geq Z_2,Z_3$ in $\mathcal{D}^-$}\\
&= 1+Z_1.
\end{split}
\end{equation}
Therefore, we have $Z_1+2Z_2+2Z_3\leq 1$, and each of $Z_1$, $Z_2$, $Z_3$ is bounded. 
\end{proof}

\subsection{Existence of ALC metrics}

\begin{proposition}
\label{prop_compact subset}
If a $\gamma_\theta^i$ with $\theta\in (0,\pi)$ enters $\mathcal{D}^\pm$, the integral curve is defined on $\mathbb{R}$. 
\end{proposition}
\begin{proof}
By \eqref{eqn_trace=1} and the Cauchy--Schwarz inequality we have $G\ge \tfrac{1}{7}$. For $\theta\in(0,\pi)$ the quantity $\frac{Z_4}{Z_1^2Z_2^2Z_3^2}$ is positive once $\gamma_\theta^i$ leaves $P_0^i$. A direct computation yields
\begin{equation}
\label{eqn_volume_derivative}
\left(\frac{Z_4}{Z_1^2Z_2^2Z_3^2}\right)'=\frac{Z_4}{Z_1^2Z_2^2Z_3^2}(1-7G)\leq 0,
\end{equation}
so $\frac{Z_4}{Z_1^2Z_2^2Z_3^2}$ is non-increasing along the integral curve. Let $\mu$ denote its value at some $\eta_*$. Then for all $\eta>\eta_*$ we have $Z_4\leq \mu Z_1^2Z_2^2Z_3^2$. 

As the integral curve enters $\mathcal{D}^\pm$, each of $Z_1$, $Z_2$, $Z_3$ is bounded by Proposition \ref{prop: bounded Z1Z2Z3}. Hence, the function $Z_4$ is also bounded along the integral curve, and so is each $X_j$ by \eqref{eqn: spin(7)}. Subsequently, the integral curve stays in a compact subset of $\mathcal{D}^\pm$ by Lemma \ref{lem: invariant+}-\ref{lem: invariant-}. Thus the integral curve is defined on $\mathbb{R}$.
\end{proof}

The critical point $P_0^{k+l}$ is on the boundary of $\mathcal{D}^+$, while $P_0^{k}$ and $P_0^{l}$ are on the boundary of $\mathcal{D}^-$. Although the inequality \eqref{eqn_new inequality} fails to hold initially for $\theta\in (0,\pi)$, we show that for a $\theta>0$ sufficiently small, the integral curve $\gamma_\theta^i$ eventually satisfies the inequality and consequently enters the corresponding invariant set at finite time. This provides a unified construction for all three families of forward complete $\spin(7)$ metrics.

\begin{lemma}
\label{lem_ALC metrics}
For each $i\in \{k+l,l,k\}$, there exists a sufficiently small $\theta_*>0$ such that each $\gamma_{\theta}^{i}$ eventually enters $\mathcal{D}^\pm$ if $\theta\in [0,\theta_*)$.
\end{lemma}
\begin{proof}
Since $Z_1=0$ and $Z_2=Z_3=\frac{1}{3}$ at $P_0^{k+l}$, the first two defining inequalities in \eqref{eqn_D+} are automatically satisfied by $\gamma_\theta^{k+l}$ for any $\theta\in [0,\pi]$. Recall that $Z_2=0$ and $Z_1=Z_3=\frac{1}{3}$ at $P_0^l$. By \eqref{eqn_linearized solution}, it is clear that $Z_1\geq Z_3$ initially along $\gamma_\theta^{l}$. Therefore, the first two defining inequalities in \eqref{eqn_D-} are satisfied by $\gamma_\theta^{l}$ for any $\theta\in [0,\pi]$. By a similar argument, the same conclusion holds for $\gamma_\theta^{k}$ for any $\theta\in [0,\pi]$. By \eqref{eqn: Z1<Z2}-\eqref{eqn: Z1<Z3} and \eqref{eqn: Z1>Z2}-\eqref{eqn: Z1>Z3}, it is clear that these inequalities are preserved even without the the third inequality in $\mathcal{D}^\pm$. Therefore, it boils down to investigating the function $\left(k\sqrt{Z_1Z_2}+l\sqrt{Z_1Z_3}+(k+l)\sqrt{Z_2Z_3}\right)Z_4$.

Fix an $\eta_\bullet\in\mathbb{R}$ and let $p^i_\bullet(\theta)=\gamma_{\theta}^{i}(\eta_\bullet)$. Recall \eqref{eqn_gamma_0^i}, each $\gamma_0^i$ is an algebraic curve that is contained in the invariant set $\mathcal{W}_i$. Take $i=k+l$ for example, we have 
\begin{equation}
\begin{split}
&\left(\left(k\sqrt{Z_1Z_2}+l\sqrt{Z_1Z_3}+(k+l)\sqrt{Z_2Z_3}\right)Z_4\right)(p^{k+l}_\bullet(0))\\
&=((k+l)Z_2Z_4)(p^{k+l}_\bullet(0))\\
&=2\Delta\left(4-\frac{1}{(Z_2)(p^{k+l}_\bullet(0))}\right)\\
&<2\Delta.
\end{split}
\end{equation}
Since the composite function is continuous, there exists a sufficiently small $\theta_*>0$ such that 
$$\left(\left(k\sqrt{Z_1Z_2}+l\sqrt{Z_1Z_3}+(k+l)\sqrt{Z_2Z_3}\right)Z_4\right)(p^{k+l}_\bullet(\theta))<2\Delta$$
if $\theta\in [0,\theta_*)$.
Therefore, the point $p^{k+l}_\bullet(\theta)$ is contained in the interior of $\mathcal{D}^+$ if $\theta\in (0,\theta_*)$, meaning that the integral curve $\gamma_\theta^{k+l}$ enters $\mathcal{D}^+$ eventually. 

The arguments for the other two cases $i\in \{k,l\}$ are similar. The proof is complete.
\end{proof}

Combining the results above, each $\gamma_\theta^{i}$ with $\theta\in(0,\theta_*)$ is defined on $\mathbb{R}$. Thus we obtain a continuous one-parameter family of forward complete $\spin(7)$ metrics on each $M_{k,l}^{i}$. We show below that these metrics are ALC.

\begin{proposition}
\label{prop_AC not in D}
Critical points $P_{AC}^\pm$ are not in the closed set $\mathcal{D}^+\cup \mathcal{D}^-$.
\end{proposition}
\begin{proof}
We first show that for any distinct $m,n\in \{1,2,3\}$, the inequality $\sqrt{Z_mZ_n}\geq \frac{1}{7}$ holds at $P_{AC}^\pm$. 

In the proof of Proposition \ref{prop: spin7_cone_chirality}, it is known that $P_{AC}^\pm$ are on the circle 
\begin{equation}
\label{eqn_circl for P_AC}
\left\{Z_1+Z_2+Z_3=\frac{4}{7}\right\}\cap \{5(Z_1^2+Z_2^2+Z_3^2)=6(Z_2Z_3+Z_3Z_1+Z_1Z_2)\}.
\end{equation}
By symmetry, it suffices to prove that $\sqrt{Z_2Z_3}\geq \frac{1}{7}$. Eliminating $Z_1$ from \eqref{eqn_circl for P_AC}, we have 
\begin{equation}
\begin{split}
0&= 16\left(Z_2+Z_3-\frac{2}{7}\right)^2+\frac{16}{49}-16Z_2Z_3\geq \frac{16}{49}-16Z_2Z_3.
\end{split}
\end{equation}
Hence $\sqrt{Z_2Z_3}\geq \frac{1}{7}$. We then have
\begin{equation}
\label{eqn_tech ineq1}
\begin{split}
(\sqrt{Z_2Z_3}-\sqrt{Z_1Z_2})^2&=Z_2(\sqrt{Z_3}-\sqrt{Z_1})^2\\
&=Z_2(Z_1+Z_3-2\sqrt{Z_1Z_3})\\
&= Z_2\left(\frac{4}{7}-Z_2-2\sqrt{Z_1Z_3}\right)\\
&\leq  Z_2\left(\frac{2}{7}-Z_2\right)\\
&\leq \frac{1}{49}.
\end{split}
\end{equation}
Similarly,
\begin{equation}
\begin{split}
\label{eqn_tech ineq2}
(\sqrt{Z_2Z_3}-\sqrt{Z_1Z_3})^2\leq \frac{1}{49}.
\end{split}
\end{equation}

For $P_{AC}^+$, we have
$
\frac{1}{7}=\frac{1}{2\Delta}((k+l)Z_2Z_3-lZ_1Z_3-kZ_1Z_2)Z_4
$
by \eqref{eqn: spin(7)}. Therefore, at $P_{AC}^+$ we have
\begin{equation}
\begin{split}
&((k+l)\sqrt{Z_2Z_3}+l\sqrt{Z_1Z_3}+k\sqrt{Z_1Z_2})Z_4\\
&=\frac{2\Delta}{7}\frac{(k+l)\sqrt{Z_2Z_3}+l\sqrt{Z_1Z_3}+k\sqrt{Z_1Z_2}}{(k+l)Z_2Z_3-lZ_1Z_3-kZ_1Z_2} \\
&=\frac{2\Delta}{7}\frac{k(\sqrt{Z_2Z_3}+\sqrt{Z_1Z_2})+l(\sqrt{Z_2Z_3}+\sqrt{Z_1Z_3})}{k(\sqrt{Z_2Z_3}-\sqrt{Z_1Z_2})(\sqrt{Z_2Z_3}+\sqrt{Z_1Z_2})+l(\sqrt{Z_2Z_3}-\sqrt{Z_1Z_3})(\sqrt{Z_2Z_3}+\sqrt{Z_1Z_3})}.
\end{split}
\end{equation}
By Proposition \ref{prop: spin7_cone_chirality}, we have  $Z_2,Z_3>Z_1$ at $P_{AC}^+$. Inequalities \eqref{eqn_tech ineq1} and \eqref{eqn_tech ineq2} becomes
$$
\sqrt{Z_2Z_3}-\sqrt{Z_1Z_2}\leq \frac{1}{7},\quad \sqrt{Z_2Z_3}-\sqrt{Z_1Z_3}\leq \frac{1}{7}.
$$
Therefore, the above quantity is hence strictly bounded below by $2\Delta$. Hence $P_{AC}^+ \notin \mathcal{D}^+ \cup \mathcal{D}^-$.

For $P_{AC}^-$, we have 
$
\frac{1}{7}=\frac{1}{2\Delta}(-(k+l)Z_2Z_3+lZ_1Z_3+kZ_1Z_2)Z_4
$
by \eqref{eqn: spin(7)}. Therefore,
\begin{equation}
\begin{split}
&((k+l)\sqrt{Z_2Z_3}+l\sqrt{Z_1Z_3}+k\sqrt{Z_1Z_2})Z_4\\
&=\frac{2\Delta}{7}\frac{(k+l)\sqrt{Z_2Z_3}+l\sqrt{Z_1Z_3}+k\sqrt{Z_1Z_2}}{-(k+l)Z_2Z_3+lZ_1Z_3+kZ_1Z_2} \\
&=\frac{2\Delta}{7}\frac{k(\sqrt{Z_2Z_3}+\sqrt{Z_1Z_2})+l(\sqrt{Z_2Z_3}+\sqrt{Z_1Z_3})}{k(\sqrt{Z_1Z_2}-\sqrt{Z_2Z_3})(\sqrt{Z_2Z_3}+\sqrt{Z_1Z_2})+l(\sqrt{Z_1Z_3}-\sqrt{Z_2Z_3})(\sqrt{Z_2Z_3}+\sqrt{Z_1Z_3})}.
\end{split}
\end{equation}
at $P_{AC}^-$. By Proposition \ref{prop: spin7_cone_chirality}, we have  $Z_2,Z_3<Z_1$ at $P_{AC}^-$. Inequalities \eqref{eqn_tech ineq1} and \eqref{eqn_tech ineq2} become
$$
\sqrt{Z_2Z_3}-\sqrt{Z_1Z_2}\geq -\frac{1}{7},\quad \sqrt{Z_2Z_3}-\sqrt{Z_1Z_3}\geq -\frac{1}{7}.
$$
The above quantity is hence strictly bounded below by $2\Delta$. Thus $P_{AC}^- \notin \mathcal{D}^+ \cup \mathcal{D}^-$.
\end{proof}

\begin{lemma}
\label{lem_ALC asymp}
If a $\gamma_\theta^i$ with $\theta\in (0,\pi)$ enters $\mathcal{D}^\pm$, the integral curve converges to $P_{ALC}$.
\end{lemma}
\begin{proof}
Recall that for $\theta\in (0,\pi)$, the quantity $\frac{Z_4}{Z_1^2Z_2^2Z_3^2}$ is non-increasing along $\gamma_\theta^i$ by \eqref{eqn_volume_derivative}. Since the integral curve is confined in a compact subset and is defined on $\mathbb{R}$ by Proposition \ref{prop_compact subset}, the limit $\lim\limits_{\eta\to\infty}\frac{Z_4}{Z_1^2Z_2^2Z_3^2}$ exists.

If $\lim\limits_{\eta\to\infty}\frac{Z_4}{Z_1^2Z_2^2Z_3^2}\neq 0$, the $\omega$-limit set is a subset of $\{G=\frac{1}{7}\}$ by \eqref{eqn_volume_derivative}. Cauchy--Schwarz inequality implies that the $\omega$-limit set is contained in $\{X_1=X_2=X_3=X_4=\frac{1}{7}\}$. Since the $\omega$-limit set is invariant, each $R_i$ in \eqref{eqn: new Einstein equation} is equal by the ODE system. With the conservation law \eqref{eqn: new conservation set}, it is known that each $Z_i$ must be constant in the $\omega$-limit set. Hence the set is contained in the finite set $\{P_{AC}^+,P_{AC}^-\}$, meaning the integral curve converges to one of the critical points. Since $P_{AC}^\pm$ are not in $\mathcal{D}^+\cup\mathcal{D}^-$ and each $\mathcal{D}^\pm$ is invariant, this leads to a contradiction.

If $\lim\limits_{\eta\to\infty}\frac{Z_4}{Z_1^2Z_2^2Z_3^2}=0$, the $\omega$-limit set is contained in $\{Z_4=0\}$, since the other $Z_j$ are bounded above. By \eqref{eqn: spin(7)}, the $\omega$-limit set is contained in $\mathcal{C}_{G_2}$. 

The dynamical subsystem restricted on $\mathcal{C}_{G_2}$ is fully known by \cite{cleyton_cohomogeneity-one_2002}. In particular, the invariant set can be parametrized by the regular triangle 
$$
\{Z_1,Z_2,Z_3\geq 0\}\cap \left\{Z_1+Z_2+Z_3=\frac{1}{2}\right\}.
$$
Each point on $\mathcal{C}_{G_2}$ belongs to the invariant algebraic curve
$$
\mathcal{L}(\xi)=\mathcal{C}_{G_2}\cap \{\cos(\xi)Z_3(Z_2 - Z_1) - \sin(\xi) Z_2(Z_3 -Z_1) =0\}
$$
for some $\quad \xi \in [0,\pi).$ If $\xi\in \left\{0,\frac{\pi}{4},\frac{\pi}{2}\right\}$, the algebraic curve is a straight line that contains $Q_0^i$, $P_{ALC}$, and $Q_1^i$. The segment that joins  $Q_0^i$ and $P_{ALC}$ is the Bryant--Salamon $G_2$ metric. For the other value of $\xi$, the algebraic curve $\mathcal{L}(\xi)$ contains $P_{ALC}$ and two of $Q_1^i$, representing a pair of singular $G_2$ metrics with an ALC limit.

The critical point $P_{ALC}\in \mathcal{C}_{G_2}$ is a sink (see \cite[(3.5)]{chi2022spin}). Furthermore, the dynamical subsystem restricted on $\mathcal{C}_{G_2}$ is fully known by \cite{cleyton_cohomogeneity-one_2002}. As discussed in the Remark \ref{rem_G2} below, each integral curve on $\mathcal{C}_{G_2}$ converges to $P_{ALC}$. Thus, the integral curve with its $\omega$-limit set in $\mathcal{C}_{G_2}$ converges to $P_{ALC}$. 
\end{proof}

\begin{remark}
\label{rem_G2}
The invariant set $\mathcal{C}_{G_2}$ can be parametrized by the regular triangle 
$$
\{Z_1,Z_2,Z_3\geq 0\}\cap \left\{Z_1+Z_2+Z_3=\frac{1}{2}\right\}.
$$
Each point on $\mathcal{C}_{G_2}$ belongs to an invariant algebraic curve
$$
\mathcal{L}(\xi)=\mathcal{C}_{G_2}\cap \{\cos(\xi)Z_3(Z_2 - Z_1) - \sin(\xi) Z_2(Z_3 -Z_1) =0\}
$$
for some $\xi \in [0,\pi).$ If $\xi\in \left\{0,\frac{\pi}{4},\frac{\pi}{2}\right\}$, the algebraic curve is a straight line that contains $Q_0^i$, $P_{ALC}$, and $Q_1^i$. The segment that joins  $Q_0^i$ and $P_{ALC}$ is the Bryant--Salamon $G_2$ metric. For the other value of $\xi$, the algebraic curve $\mathcal{L}(\xi)$ contains $P_{ALC}$ and two of $Q_1^i$, representing a pair of singular $G_2$ metrics with an ALC limit.
\end{remark}

\section{Invariant set where the circle fiber blows up}
\label{sec: invariant set 2}
In this section, we construct invariant sets where integral curves do not converge to any of the critical points of the $\spin(7)$ system. Geometrically, integral curves entering these sets correspond to metrics whose principal curvature of the $\mathbb{S}^1$-fiber in $N_{k,l}$ dominates the remaining principal curvatures beyond a fixed threshold. Once this threshold is exceeded, the $\mathbb{S}^1$-fiber expands too rapidly for the metric to exhibit either ALC or AC asymptotics. 

\begin{lemma}
\label{lem_B is bad}
Define 
$$\mathcal{B}^\pm:=\mathcal{C}^\pm_{\spin(7)}\cap \left\{X_4-X_1-X_2-X_3\geq 0\right\}.$$
The sets $\mathcal{B}^\pm$ are invariant.
\end{lemma}
\begin{proof}
Consider 
\begin{equation}
\label{eqn_comp B}
\begin{split}
&\left.(X_4-X_1-X_2-X_3)'\right|_{X_4-X_1-X_2-X_3=0}\\
&=(X_4-X_1-X_2-X_3)(G-1)+R_4-R_1-R_2-R_3\\
&=\left(\frac{(k+l)^2}{\Delta^2}Z_2^2Z_3^2+\frac{l^2}{\Delta^2}Z_1^2Z_3^2+\frac{k^2}{\Delta^2}Z_1^2Z_2^2\right)Z_4^2\\
&\quad +Z_1^2+Z_2^2+Z_3^2-6Z_1Z_2-6Z_2Z_3-6Z_1Z_3
\end{split}
\end{equation}

By \eqref{eqn_trace=1} and \eqref{eqn: conservation Z spin(7) 2}, the equation $X_4-X_1-X_2-X_3=0$ is equivalent to $Z_1+Z_2+Z_3=\frac{2}{3}$. By \eqref{eqn: conservation Z spin(7) 1}, we have
$$Z_4^2=\left(\frac{2(Z_1+Z_2+Z_3)-1}{\frac{k+l}{2\Delta}Z_2Z_3-\frac{l}{2\Delta}Z_1Z_3-\frac{k}{2\Delta}Z_1Z_2}\right)^2=\left(\frac{1}{2}\frac{Z_1+Z_2+Z_3}{\frac{k+l}{2\Delta}Z_2Z_3-\frac{l}{2\Delta}Z_1Z_3-\frac{k}{2\Delta}Z_1Z_2}\right)^2.$$ The computation \eqref{eqn_comp B} becomes
\begin{equation}
\label{eqn_x4-x1-x2-x3}
\begin{split}
&\left.(X_4-X_1-X_2-X_3)'\right|_{X_4-X_1-X_2-X_3=0}\\
&=\left(\frac{(k+l)^2}{\Delta^2}Z_2^2Z_3^2+\frac{l^2}{\Delta^2}Z_1^2Z_3^2+\frac{k^2}{\Delta^2}Z_1^2Z_2^2\right)\left(\frac{1}{2}\frac{Z_1+Z_2+Z_3}{\frac{k+l}{2\Delta}Z_2Z_3-\frac{l}{2\Delta}Z_1Z_3-\frac{k}{2\Delta}Z_1Z_2}\right)^2\\
&\quad +Z_1^2+Z_2^2+Z_3^2-6Z_1Z_2-6Z_2Z_3-6Z_1Z_3\\
&=\frac{2\Phi}{\left((k+l)Z_2Z_3-lZ_1Z_3-kZ_1Z_2\right)^2},
\end{split}
\end{equation}
where 
\begin{equation}
\begin{split}
\Phi &=\Phi_2l^2+\Phi_1kl+\Phi_0k^2\\
\Phi_2&=Z_1^{4} Z_3^{2}-3 Z_1^{3} Z_2 Z_3^{2}-2 Z_1^{3} Z_3^{3}+8 Z_1^{2} Z_2^{2} Z_3^{2}+4 Z_1^{2} Z_2 Z_3^{3}+Z_1^{2} Z_3^{4}-3 Z_1 Z_2^{3} Z_3^{2}\\
&\quad +4 Z_1 Z_2^{2} Z_3^{3}-Z_1 Z_2Z_3^{4}+Z_2^{4} Z_3^{2}-2 Z_2^{3} Z_3^{3}+Z_2^{2} Z_3^{4},\\
\Phi_1&=Z_1^{4} Z_2 Z_3-7 Z_1^{3} Z_2^{2} Z_3-7 Z_1^{3} Z_2 Z_3^{2}+7 Z_1^{2} Z_2^{3} Z_3+8 Z_1^{2} Z_2^{2} Z_3^{2}\\
&\quad +7 Z_1^{2} Z_2 Z_3^{3}-Z_1 \,Z_2^{4} Z_3+Z_1 Z_2^{3} Z_3^{2}+Z_1Z_2^{2} Z_3^{3}-Z_1 Z_2 Z_3^{4}+2 Z_2^{4} Z_3^{2}-4 Z_2^{3} Z_3^{3}+2 Z_2^{2} Z_3^{4},\\
\Phi_0&=Z_1^{4} Z_2^{2}-2 Z_1^{3} Z_2^{3}-3 Z_1^{3} Z_2^{2} Z_3+Z_1^{2} Z_2^{4}+4 Z_1^{2} Z_2^{3} Z_3+8 Z_1^{2} Z_2^{2} Z_3^{2}-Z_1 Z_2^{4} Z_3\\
&\quad +4 Z_1 Z_2^{3} Z_3^{2}-3 Z_1 Z_2^{2} Z_3^{3}+Z_2^{4} Z_3^{2}-2 Z_2^{3} Z_3^{3}+Z_2^{2} Z_3^{4}.
\end{split}
\end{equation}

Since
\begin{equation}
\label{eqn_tricky discrim}
\begin{split}
\delta&=\Phi_1^2-4\Phi_0\Phi_2\\
&=-3 Z_1^{2} Z_2^{2} Z_3^{2} (Z_1+Z_2+Z_3)^{2} \left(Z_1^{2}+Z_2^{2}+Z_3^{2}-2 Z_2 Z_3-2 Z_1 Z_3-2 Z_1 Z_2\right)^{2}\\
&\leq 0,
\end{split}
\end{equation}
the polynomial $\Phi$ is non-negative if $\Phi_2$ or $\Phi_0$ is positive.

We proceed to show that $\Phi_2 \geq 0$ and its zero set in $\{X_4-X_1-X_2-X_3=0\}$ (equivalently $\{Z_1+Z_2+Z_3=\frac{2}{3}\}$) is $\{Z_3 = 0\}\cup\{P_0^{k+l},P_0^l\}$. Define $a=\frac{Z_1}{Z_3}$ and $b=\frac{Z_2}{Z_3}$. Rewrite $\Phi_2=Z_3^6\phi_2$, where
\begin{equation}
\phi_2(a,b)=a^4-3a^3b-2a^3+8a^2b^2+4a^2b+a^2-3ab^3+4ab^2-ab+b^4-2b^3+b^2.
\end{equation}

Since
\begin{equation}
\begin{split}
\frac{\partial \phi_2}{\partial a}+\frac{\partial \phi_2}{\partial b}&= a(a-1)^2+b(b-1)^2 + ab(7a + 7b  + 16)\geq 0,\\
\frac{\partial \phi_2}{\partial a}-\frac{\partial \phi_2}{\partial b}&=(a-b)(7a^2 - 18ab + 7b^2 - 10a - 10b + 3),
\end{split}
\end{equation}
the critical points of $\phi_2$ in $[0,\infty)\times [0,\infty)$ are $(0,0)$, $(1,0)$, or $(0,1)$.
The Hessian of $\phi_2$ at these points are
$$
\mathrm{Hess}\phi_2(0,0)=\begin{bmatrix}
2&-1\\
-1&2
\end{bmatrix},\quad \mathrm{Hess}\phi_2(1,0)=\begin{bmatrix}
2&-2\\
-2&26
\end{bmatrix},\quad \mathrm{Hess}\phi_2(0,1)=\begin{bmatrix}
26&-2\\
-2&2
\end{bmatrix}.
$$
Hence, these points are local minima. Since $\phi_2$ vanishes at these point, the function $\phi_2\geq 0$ on $[0,\infty)\times [0,\infty)$ and it only vanishes at $(0,0)$, $(1,0)$, and $(0,1)$. 

By \eqref{eqn_trace=1}, the equation $X_4-X_1-X_2-X_3=0$ is equivalent to $X_4=\frac{1}{3}$. By \eqref{eqn: spin(7)}, at most one of $Z_1$, $Z_2$, $Z_3$ vanishes. Threfore, we can ignore $(0,0)$. Correspondingly, the zero set of $\Phi_2=Z_3^6\phi_2$ zero set in $\{X_4-X_1-X_2-X_3=0\}$ is $\{Z_3 = 0\}\cup \{P_0^{l},P_0^{k+l}\}$. By the symmetry $(l,Z_2)\leftrightarrow (k,Z_3)$, the function $\Phi_0$ is non-negative and its zero set is $\{Z_2 = 0\}\cup \{P_0^{k},P_0^{k+l}\}$. 

By our analysis above, both $\Phi_2$ and $\Phi_0$ vanish on $X_4-X_1-X_2-X_3=0$ only at the critical point $P_0^{k+l}$. As $\Phi_1$ vanishes at $P_0^{k+l}$, we conclude that the polynomial $\Phi$, and hence the derivative \eqref{eqn_x4-x1-x2-x3} is non-negative. 

By \eqref{eqn_tricky discrim}, If a non-transversal intersection occurs in $\{Z_1+Z_2+Z_3=\frac{2}{3}\}$, it is necessary that $\delta=0$ at the intersection point. If $Z_1Z_2Z_3=0$, we obtain one of $P_0^i$ from $\Phi=0$. Otherwise, we have 
$$
Z_1^2+Z_2^2+Z_3^2=2Z_2Z_3+2Z_1Z_3+2Z_1Z_2.
$$
With $Z_1^2+Z_2^2+Z_3^2+2Z_2Z_3+2Z_1Z_3+2Z_1Z_2=(Z_1+Z_2+Z_3)^2=\frac{4}{9}$, it is clear that 
$$Z_1^2+Z_2^2+Z_3^2=2Z_2Z_3+2Z_1Z_3+2Z_1Z_2=\frac{2}{9}.$$
Therefore, at the non-transversal intersection point, the derivative \eqref{eqn_x4-x1-x2-x3} becomes
\begin{equation}
\begin{split}
&\left.(X_4-X_1-X_2-X_3)'\right|_{X_4-X_1-X_2-X_3=0}\\
&=(X_4-X_1-X_2-X_3)(G-1)+R_4-R_1-R_2-R_3\\
&=\left(\frac{(k+l)^2}{\Delta^2}Z_2^2Z_3^2+\frac{l^2}{\Delta^2}Z_1^2Z_3^2+\frac{k^2}{\Delta^2}Z_1^2Z_2^2\right)\left(\frac{1}{2}\frac{Z_1+Z_2+Z_3}{\frac{k+l}{2\Delta}Z_2Z_3-\frac{l}{2\Delta}Z_1Z_3-\frac{k}{2\Delta}Z_1Z_2}\right)^2\\
&\quad +Z_1^2+Z_2^2+Z_3^2-6Z_1Z_2-6Z_2Z_3-6Z_1Z_3\\
&=\left(\frac{(k+l)^2}{\Delta^2}Z_2^2Z_3^2+\frac{l^2}{\Delta^2}Z_1^2Z_3^2+\frac{k^2}{\Delta^2}Z_1^2Z_2^2\right)\left(\frac{1}{3}\frac{1}{\frac{k+l}{2\Delta}Z_2Z_3-\frac{l}{2\Delta}Z_1Z_3-\frac{k}{2\Delta}Z_1Z_2}\right)^2-\frac{4}{9}\\
&=\frac{8}{9}\frac{Z_1Z_2Z_3}{\left((k+l)Z_2Z_3-lZ_1Z_3-kZ_1Z_2\right)^2}\left((k+l)lZ_3+(k+l)kZ_2-lkZ_1\right).
\end{split}
\end{equation}
The non-transversal intersection point in $\{Z_1+Z_2+Z_3=\frac{2}{3}\}$ with $Z_1Z_2Z_3\neq 0$ is hence characterized by the following set of equations
\begin{equation}
\begin{split}
Z_1+Z_2+Z_3&=\frac{2}{3},\\
Z_1^2+Z_2^2+Z_3^2&=2Z_2Z_3+2Z_1Z_3+2Z_1Z_2,\\
(k+l)lZ_3+(k+l)kZ_2-lkZ_1&=0,
\end{split}
\end{equation}
whose only solution is $P_*=(-\frac{kl}{3\Delta},\frac{(k+l)k}{3\Delta},\frac{(k+l)l}{3\Delta},\frac{1}{3},\frac{(k+l)^2}{3\Delta},\frac{l^2}{3\Delta},\frac{k^2}{3\Delta},\frac{6\Delta^2}{kl(k+l)})\in \mathcal{C}^-_{\spin(7)}$. By continuous dependence, the integral curve that starts at $P_*$ also enters the interior of $\mathcal{B}^-$. Therefore, both sets $\mathcal{B}^\pm$ are invariant.
\end{proof}

Recall that each $\gamma_\pi^i$ lies on the algebraic curve $\mathcal{W}_i$ in \eqref{eqn_W curve}, where one of the $Z_j$ vanishes identically while the other two remain greater than $\frac{1}{3}$. Hence, these integral curves lie in $\mathcal{B}^+\cup \mathcal{B}^-$. By \eqref{eqn_linearized solution} and Lemma \ref{lem_B is bad}, each $\gamma_\theta^i$ stays in $\mathcal{B}^+\cup \mathcal{B}^-$ initially whenever $\theta\in\left(\pi-\arctan\left(\frac{1}{3i}\right),\pi\right]$.

For each $\gamma_\theta^i$, define 
$$\theta_i:=\sup \{\theta \in [0,\pi]\mid \gamma_\theta^i\text{ eventually enters }\mathcal{D}^+\cup \mathcal{D}^-\}.$$

\begin{lemma}
\label{lem_AC metrics}
The integral curve $\gamma_{\theta_{k+l}}^{k+l}$ converges to $P_{AC}^+$. The integral curves $\gamma_{\theta_l}^{l}$ and $\gamma_{\theta_k}^{k}$ converge to $P_{AC}^-$.
\end{lemma}
\begin{proof}
By Lemma \ref{lem_ALC metrics}, we have $\theta_i\geq \theta_i^*>0$.  By our discussion above, we also have $\theta_i\leq \pi-\arctan\left(\frac{1}{3i}\right)<\pi$. Hence $\theta_i\in(0,\pi)$.

Suppose a $\gamma_{\theta_i}^i$ eventually enters $\mathcal{D}^+\cup\mathcal{D}^-$ in finite time. By continuous dependence, the integral curve $\gamma_{\theta_i+\epsilon}^i$ eventually enters $\mathcal{D}^+\cup\mathcal{D}^-$ in finite time for a sufficiently small $\epsilon>0$. Since $\mathcal{D}^+\cup\mathcal{D}^-$ is invariant by Lemma \ref{lem: invariant+}-\ref{lem: invariant-}, this contradicts the definition of $\theta_i$. 

On the other hand, suppose $\gamma_{\theta_i}^i$ eventually enters $\mathcal{B}^+\cup\mathcal{B}^-$ in finite time. By continuous dependence, the integral curve $\gamma_{\theta_i-\epsilon}^i$ eventually enters $\mathcal{B}^+\cup\mathcal{B}^-$ in finite time for a sufficiently small $\epsilon>0$. As $\mathcal{B}^+\cup\mathcal{B}^-$ is invariant by Lemma \ref{lem_B is bad}, this also contradicts the definition of $\theta_i$. Hence, each $\gamma_{\theta_i}^i$ stays in the region between $\mathcal{D}^+\cup\mathcal{D}^-$ and $\mathcal{B}^+\cup\mathcal{B}^-$. In particular, the integral curve $\gamma_{\theta_{k+l}}^{k+l}$ remains in 
$$
\mathcal{A}^+:=\mathcal{C}_{\spin(7)}^+\cap \left\{\left(k\sqrt{Z_1Z_2}+l\sqrt{Z_1Z_3}+(k+l)\sqrt{Z_2Z_3}\right)Z_4\geq 2\Delta\right\}\cap \left\{Z_1+Z_2+Z_3\leq \frac{2}{3}\right\}.
$$
By the boundedness of $Z_j$ for $j\in\{1,2,3\}$ and \eqref{eqn_volume_derivative}, we know that $\gamma_{\theta_{k+l}}^{k+l}$ is in a compact subset of $\mathcal{A}^+$. Hence, the integral curve is defined on $\mathbb{R}$. Furthermore, by \eqref{eqn_volume_derivative} and the first defining inequality of $\mathcal{A}^+$, the function $\frac{Z_4}{Z_1^2Z_2^2Z_3^2}$ decreases to some positive limit. Hence $G\rightarrow \frac{1}{7}$ along $\gamma_{\theta_{k+l}}^{k+l}$. Cauchy--Schwarz inequality implies that the $\omega$-limit set of $\gamma_{\theta_{k+l}}^{k+l}$ is contained in $\{X_1=X_2=X_3=X_4=\frac{1}{7}\}$. Since the $\omega$-limit set is invariant, each $R_i$ is equal by the ODE system \eqref{eqn: new Einstein equation}. With the conservation law \eqref{eqn: new conservation set}, it is known that each $Z_i$ must be constant in the $\omega$-limit set. Therefore, the $\omega$-limit set of $\gamma_{\theta_{k+l}}^{k+l}$ is contained in $\{P_{AC}^+, P_{AC}^-\}$. By Proposition \ref{prop: spin7_cone_chirality} and Proposition \ref{prop_AC not in D}, the only critical points in $\mathcal{A}^+$ are $P_0^{k+l}$ and $P_{AC}^+$. Hence, the integral curve converges to $P_{AC}^+$.

Analogously, the integral curves $\gamma_{\theta_l}^{\,l}$ and $\gamma_{\theta_k}^{\,k}$ remain in
$$
\mathcal{A}^-:=\mathcal{C}_{\spin(7)}^-\cap \left\{\left(k\sqrt{Z_1Z_2}+l\sqrt{Z_1Z_3}+(k+l)\sqrt{Z_2Z_3}\right)Z_4\geq 2\Delta\right\}\cap \left\{Z_1+Z_2+Z_3\leq \frac{2}{3}\right\},
$$ 
and their $\omega$-limit sets lie in $\{P_{AC}^+,P_{AC}^-\}$.
By Proposition \ref{prop: spin7_cone_chirality} and Proposition \ref{prop_AC not in D},
the only critical points in $\mathcal{A}^-$ are $P_0^l$, $P_0^k$, and $P_{AC}^-$.  Therefore, both integral curves converge to $P_{AC}^-$.
\end{proof}

Theorem \ref{thm: main 1} is established by Lemma \ref{lem_ALC metrics}, Lemma \ref{lem_ALC asymp} and Lemma \ref{lem_AC metrics}.

We conclude by mentioning conically singular $\spin(7)$ metrics, whose principal orbit collapses to a point as $t\to 0$. Near the singular end these metrics are asymptotic to the Euclidean cone over the homogeneous Einstein $N_{k,l}$, while at infinity they exhibit ALC asymptotics. These metrics are represented by integral curves that emanate from $P_{AC}^\pm$. Such examples are already known in exceptional cases: for $N_{1,1}$, this includes the $\mathbb{A}_8$ metric in \cite{cvetic_cohomogeneity_2002} and the $\Gamma_0$ metric in \cite[Theorem 1.3]{chi_einstein_2020}. For $N_{1,0}$, the conically singular metric was constructed in \cite{lehmann2022geometric}. It is natural to expect that similar solutions should exist for general coprime pairs $(k,l)$. The main obstacle toward a uniform proof appears to be computational: the coordinates of $P_{AC}^\pm$ are determined by quartic equations. It would be interesting to understand whether these computational difficulties can be bypassed, perhaps by a more conceptual approach.

\section*{Acknowledgements}
The author thanks McKenzie Wang and Lorenzo Foscolo for their valuable discussions and helpful comments.

\section*{Conflict of Interest}
The author declares that there is no conflict of interest regarding the publication of this paper.

\section*{Data Availability}
Data sharing is not applicable as no datasets were generated or analysed for the present article.

\bibliography{BIB}
\bibliographystyle{alpha}

\end{document}